\documentclass[a4paper, reqno, oneside]{amsart}

\usepackage[latin1]{inputenc}
\usepackage{amsmath}
\usepackage{amsfonts}
\usepackage{amssymb}
\usepackage{graphicx}
\usepackage{enumerate}
\usepackage{mathrsfs}

\usepackage[colorlinks=true,citecolor=black,linkcolor=black,urlcolor=blue]{hyperref}

\newtheorem{theorem}{Theorem}[section]

\newtheorem{conjecture}[theorem]{Conjecture}

\theoremstyle{plain}

\newtheorem{case}{Case}

\numberwithin{subcase}{case}

\newtheorem{claim}[theorem]{Claim}

\newtheorem{question}{Question}

\newtheorem{corollary}[theorem]{Corollary}

\newtheorem{lemma}[theorem]{Lemma}

\newtheorem{proposition}[theorem]{Proposition}

\newtheorem*{remark}{Remark}

\newcommand{\arxiv}[1]{\href{http://arxiv.org/abs/#1}{\texttt{arXiv:#1}}}

\begin{document}

\author{Karl Heuer}
\address{Karl Heuer, Department of Applied Mathematics and Computer Science, Technical University of Denmark, Richard Petersens Plads, building 322, 2800 Kongens Lyngby, Denmark}

\title[]{Hamiltonicity in locally finite graphs: two extensions and a counterexample}

\begin{abstract}
We state a sufficient condition for the square of a locally finite graph to contain a Hamilton circle, extending a result of Harary and Schwenk about finite graphs.

We also give an alternative proof of an extension to locally finite graphs of the result of Chartrand and Harary that a finite graph not containing $K^4$ or $K_{2,3}$ as a minor is Hamiltonian if and only if it is $2$-connected.
We show furthermore that, if a Hamilton circle exists in such a graph, then it is unique and formed by the $2$-contractible edges.

The third result of this paper is a construction of a graph which answers positively the question of Mohar whether regular infinite graphs with a unique Hamilton circle exist.
\end{abstract}

\maketitle

\section{Introduction}

Results about Hamilton cycles in finite graphs can be extended to locally finite graphs in the following way.
For a locally finite connected graph $G$ we consider its Freudenthal compactification $|G|$~\cite{diestel_buch, diestel_arx}. This is a topological space obtained by taking~$G$, seen as a $1$-complex, and adding certain points to it.
These additional point are the \textit{ends} of $G$, which are the equivalence classes of the rays of $G$ under the relation of being inseparable by finitely many vertices.
Extending the notion of cycles, we define \textit{circles}~\cite{inf-cyc-1, inf-cyc-2} in $|G|$ as homeomorphic images of the unit circle $S^1 \subseteq \mathbb{R}^2$ in $|G|$, and we call them \textit{Hamilton circles} of $G$ if they contain all vertices of $G$.
As a consequence of being a closed subspace of $|G|$, Hamilton circles also contain all ends of $G$.
Following this notion we call $G$ \textit{Hamiltonian} if there is a Hamilton circle in $|G|$.

One of the first and probably one of the deepest results about Hamilton circles was Georgakopoulos's extension of Fleischner's theorem to locally finite graphs.

\begin{theorem}\label{fin_fleisch}\cite{fleisch}
The square of any finite $2$-connected graph is Hamiltonian.
\end{theorem}

\begin{theorem}\label{inf-fleisch}\cite[Thm.\ 3]{agelos}
The square of any locally finite $2$-connected graph is Hamiltonian.
\end{theorem}

\noindent Following this breakthrough, more Hamiltonicity theorems have been extended to locally finite graphs in this way \cite{brewster-funk, bruhn-HC, agelos, Ha_Leh_Po, heuer_ObSu, heuer_Asra, lehner-HC}.

The purpose of this paper is to extend two more Hamiltonicity results about finite graphs to locally finite ones and to construct a graph which shows that another result does not extend.

The first result we consider is a corollary of the following theorem of Harary and Schwenk.
A \textit{caterpillar} is a tree such that after deleting its leaves only a path is left.
Let $S(K_{1, 3})$ denote the graph obtained by taking the star with three leaves, $K_{1, 3}$, and subdividing each edge once.

\begin{theorem}\label{fin_cater}\cite[Thm.\ 1]{cat_HC}
Let $T$ be a finite tree with at least three vertices. Then the following statements are equivalent:
\begin{enumerate}[\normalfont(i)]
\item $T^2$ is Hamiltonian.
\item $T$ does not contain $S(K_{1,3})$ as a subgraph.
\item $T$ is a caterpillar.
\end{enumerate}
\end{theorem}

Theorem~\ref{fin_cater} has the following obvious corollary.

\begin{corollary}\label{fin_cater_impl}~\cite{cat_HC}
The square of any finite graph $G$ on at least three vertices such that $G$ contains a spanning caterpillar is Hamiltonian.
\end{corollary}

While the proof of Corollary~\ref{fin_cater_impl} is immediate, the proof of the following extension of it, which is the first result of this paper, needs more work.
We call the closure $\overline{H}$ in $|G|$ of a subgraph $H$ of $G$ a \textit{standard subspace} of $|G|$.
Extending the notion of trees, we define \textit{topological trees} as topologically connected standard subspaces not containing any circles.
As an analogue of a path, we define an \textit{arc} as a homeomorphic image of the unit interval $[0,1] \subseteq \mathbb{R}$ in $|G|$.
Note that for standard subspaces being topologically connected is equivalent to being arc-connected by Lemma~\ref{arc_conn}.
For our extension we adapt the notion of a caterpillar to the space $|G|$ and work with \textit{topological caterpillars}, which are topological trees $\overline{T}$ such that $\overline{T-L}$ is an arc, where $T$ is a forest in $G$ and $L$ denotes the set of vertices of degree $1$ in~$T$.

\begin{theorem}\label{top_catp_HC}
The square of any locally finite connected graph $G$ on at least three vertices such that $|G|$ contains a spanning topological caterpillar is Hamiltonian.
\end{theorem}

The other two results of this paper concern the uniqueness of Hamilton circles.
The first is about finite \textit{outerplanar graphs}.
These are finite graphs that can be embedded in the plane so that all vertices lie on the boundary of a common face.
Clearly, finite outerplanar graphs have a Hamilton cycle if and only if they are $2$-connected.
In a $2$-connected graph call an edge $2$-\textit{contractible} if its contraction leaves the graph $2$-connected.
It is also easy to see that any finite $2$-connected outerplanar graph has a unique Hamilton cycle.
This cycle consists precisely of the $2$-contractible edges of the graph (except for the $K^3$), as pointed out by Sys\l{}o~\cite{syslo}.
We summarise this with the following proposition.

\begin{proposition}\label{summary}
\begin{enumerate}[\normalfont(i)]
\item A finite outerplanar graph is Hamiltonian if and only if it is $2$-connected.
\item \cite[Thm.~6]{syslo} Finite $2$-connected outerplanar graphs have a unique Hamilton cycle, which consists precisely of the $2$-contractible edges unless the graph is isomorphic to a $K^3$.
\end{enumerate}
\end{proposition}

Finite outerplanar graphs can also be characterised by forbidden minors, which was done by Chartrand and Harary.

\begin{theorem}\label{outerplanar_count_char}\cite[Thm.\ 1]{char_har}
A finite graph is outerplanar if and only if it contains neither a $K^4$ nor a $K_{2, 3}$ as a minor.\footnote{\label{stronger_K_4}Actually these statements can be strengthened a little bit by replacing the part about not containing a $K^4$ as a minor by not containing it as a subgraph.
This follows from Lemma~\ref{K^4_minor_subgr}.}
\end{theorem}

In the light of Theorem~\ref{outerplanar_count_char} we first prove the following extension of statement~(i) of Proposition~\ref{summary} to locally finite graphs.

\begin{theorem}\label{HC_K_4-K_2,3}
Let $G$ be a locally finite connected graph.
Then the following statements are equivalent:
\begin{enumerate}[\normalfont(i)]
\item $G$ is $2$-connected and contains neither $K^4$ nor $K_{2,3}$ as a minor.$^{\ref{stronger_K_4}}$
\item $|G|$ has a Hamilton circle $C$ and there exists an embedding of $|G|$ into a closed disk such that $C$ is mapped onto the boundary of the disk.
\end{enumerate}
Furthermore, if statements (i) and (ii) hold, then $|G|$ has a unique Hamilton circle.
\end{theorem}

\noindent From this we then obtain the following corollary, which extends statement (ii) of Proposition~\ref{summary}.

\begin{corollary}\label{Cor_contr}
Let $G$ be a locally finite $2$-connected graph not containing $K^4$ or $K_{2,3}$ as a minor, and not isomorphic to $K^3$.
Then the edges contained in the Hamilton circle of $|G|$ are precisely the $2$-contractible edges of $G$.
\end{corollary}

We should note here that parts of Theorem~\ref{HC_K_4-K_2,3} and Corollary~\ref{Cor_contr} are already known.
Chan~\cite[Thm.~20 with Thm.~27]{chan} proved that a locally finite $2$-connected graph not isomorphic to $K^3$ and not containing $K^4$ or $K_{2,3}$ as a minor has a Hamilton circle that consists precisely of the $2$-contractible edges of the graph.
He deduces this from other general results about $2$-contractible edges in locally finite $2$-connected graphs.
In our proof, however, we directly construct the Hamilton circle and show its uniqueness without working with $2$-contractible edges.
Afterwards, we deduce Corollary~\ref{Cor_contr}.

Our third result is related to the following conjecture Sheehan made for finite graphs.

\begin{conjecture}\label{sheehan}\cite{sheehan}
There is no finite $r$-regular graph with a unique Hamilton cycle for any $r > 2$.
\end{conjecture}

\noindent This conjecture is still open, but some partial results have been proved \cite{haxell, thomason, thomassen}.
For $r=3$ the statement of the conjecture was first verified by C.~A.~B.~Smith.
This was noted in an article of Tutte~\cite{tutte} where the statement for $r=3$ was published for the first time.

For infinite graphs Conjecture~\ref{sheehan} is not true in this formulation.
It fails already with $r=3$.
To see this consider the graph depicted in Figure~\ref{double_ladder}, called the \textit{double ladder}.

\begin{figure}[htbp]
\centering
\includegraphics{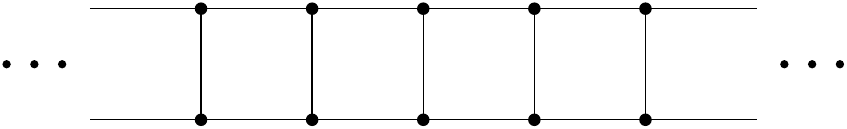}
\caption{The double ladder}
\label{double_ladder}
\end{figure}

\noindent It is easy to check that the double ladder has a unique Hamilton circle, but all vertices have degree $3$.
Mohar has modified the statement of the conjecture and raised the following question.
To state them we need to define two terms.
We define the \textit{vertex-} or \textit{edge-degree} of an end $\omega$ to be the supremum of the number of vertex- or edge-disjoint rays in $\omega$, respectively.
In particular, ends of a graph $G$ can have infinite degree, even if $G$ is locally finite.
\vspace{6pt}

\begin{question}\label{Q2}~\cite{mohar}
Does an infinite graph exist that has a unique Hamilton circle and degree $r > 2$ at every vertex as well as vertex-degree $r$ at every end?
\end{question}

Our result shows that, in contrast to Conjecture~\ref{sheehan} and its known cases, there are infinite graphs having the same degree at every vertex and end while being Hamiltonian in a unique way.

\begin{theorem}\label{Q_mohar_Yes}
There exists an infinite connected graph $G$ with a unique Hamilton circle that has degree $3$ at every vertex and vertex- as well as edge-degree $3$ at every end.
\end{theorem}

\noindent So with Theorem~\ref{Q_mohar_Yes} we answer Question~\ref{Q2} positively and, therefore, disprove the modified version of Conjecture~\ref{sheehan} for infinite graphs in the way Mohar suggested by considering degrees of both, vertices and ends.

The rest of this paper is structured as follows.
In Section~2 we establish all necessary notation and terminology for this the paper.
We also list some lemmas that will serve as auxiliary tools for the proofs of the main theorems.
Section~3 is dedicated to Theorem~\ref{top_catp_HC} where at the beginning of that section we discuss how one can sensibly extend Corollary~\ref{fin_cater_impl} and which problems arise when we try to extend Theorem~\ref{fin_cater} in a similar way.
In Section~4 we present a proof of Theorem~\ref{HC_K_4-K_2,3}.
Afterwards we describe how a different proof of this theorem works which copies the ideas of a proof of statement~(i) of Proposition~\ref{summary}.
We conclude this section by comparing the two proofs.
The last section, Section~5, contains the construction of a graph witnessing Theorem~\ref{Q_mohar_Yes}.

\section{Preliminaries}

When we mention a graph in this paper we always mean an undirected and simple graph.
For basic facts and notation about finite as well as infinite graphs we refer the reader to \cite{diestel_buch}.
For a broader survey about locally finite graphs and a topological approach to them see \cite{diestel_arx}.

Now we list important notions and concepts that we shall need in this paper followed by useful statements about them.
In a graph $G$ with a vertex $v$ we denote by $\delta(v)$ the set of edges incident with $v$ in $G$.
Similarly, for a subgraph $H$ of $G$ or just its vertex set we denote by $\delta(H)$ the set of edges that have only one endvertex in $H$.
Although formally different, we will not always distinguish between a cut $\delta(H)$ and the partition $(V(H), V(G) \setminus V(H))$ it is induced by.
For two vertices $v, w \in V(G)$ let $d_G(v,w)$ denote the distance between $v$ and $w$ in $G$.

We call a finite graph \textit{outerplanar} if it can be embedded in the plane such that all vertices lie on the boundary of a common face.

For a graph $G$ and an integer $k \geq 2$ we define the $k$\textit{-th power} of $G$ as the graph obtained by taking $G$ and adding additional edges $vw$ for any two vertices $v, w \in V(G)$ such that $1 < d_G(v, w) \leq k$.

A tree is called a \textit{caterpillar} if after the deletion of its leaves only a path is left.

We denote by $S(K_{1, 3})$ the graph obtained by taking the star with three leaves $K_{1, 3}$ and subdividing each edge once.

We call a graph \textit{locally finite} if each vertex has finite degree.

A one-way infinite path in a graph $G$ is called a \textit{ray} of $G$, while we call a two-way infinite path in $G$ a \textit{double ray} of $G$.
Every ray contains a unique vertex that has degree $1$ it.
We call this vertex the \textit{start vertex} of the ray.
An equivalence relation can be defined on the set of rays of a graph $G$ by saying that two rays are equivalent if and only if they cannot be separated by finitely many vertices in $G$.
The equivalence classes of this relation are called the \textit{ends} of $G$.
We denote the set of all ends of a graph $G$ by $\Omega(G)$.

The union of a ray $R$ with infinitely many disjoint paths $P_i$ for $i \in \mathbb{N}$ each having precisely one endvertex on $R$ is called a \textit{comb}.
We call the endvertices of the paths $P_i$ that do not lie on $R$ and those vertices $v$ for which there is a $j \in \mathbb{N}$ such that $v = P_j$ the \textit{teeth} of the comb.

The following lemma is a basic tool for infinite graphs.
Especially for locally finite graphs it helps us to get a comb whose teeth lie in a previously fixed infinite set of vertex.

\begin{lemma}\label{star-comb}\cite[Prop.\ 8.2.2]{diestel_buch}
Let $U$ be an infinite set of vertices in a connected graph $G$.
Then $G$ contains either a comb with all teeth in $U$ or a subdivision of an infinite star with all leaves in $U$.
\end{lemma}

For a locally finite and connected graph $G$ we can endow $G$ together with its ends with a topology that yields the space $|G|$.
A precise definition of $|G|$ can be found in \cite[Ch.\ 8.5]{diestel_buch}.
Let us point out here that a ray of $G$ converges in $|G|$ to the end of $G$ it is contained in.
Another way of describing $|G|$ is to endow $G$ with the topology of a $1$-complex and then forming the Freudenthal compactification \cite{freud-equi}.

For a point set $X$ in $|G|$, we denote its closure in $|G|$ by $\overline{X}$.
We shall often write $\overline{M}$ for some $M$ that is a set of edges or a subgraph of $G$.
In this case we implicitly assume to first identify $M$ with the set of points in $|G|$ which corresponds to the edges and vertices that are contained in $M$.

We call a subspace $Z$ of $|G|$ \textit{standard} if $Z = \overline{H}$ for a subgraph $H$ of $G$.

A \textit{circle} in $|G|$ is the image of a homeomorphism having the unit circle $S^1$ in $\mathbb{R}^2$ as domain and mapping into $|G|$.
Note that all finite cycles of a locally finite connected graph $G$ correspond to circles in $|G|$, but there might also be infinite subgraphs $H$ of $G$ such that $\overline{H}$ is a circle in $|G|$.
Similar to finite graphs we call a locally finite connected graph $G$ \textit{Hamiltonian} if there exists a circle in $|G|$ which contains all vertices of $G$.
Such circles are called \textit{Hamilton circles} of $G$.

We call the image of a homeomorphism with the closed real unit interval $[0, 1]$ as domain and mapping into $|G|$ an \textit{arc} in $|G|$.
Given an arc $\alpha$ in $|G|$, we call a point $x$ of $|G|$ an \textit{endpoint} of $\alpha$ if $0$ or $1$ is mapped to $x$ by the homeomorphism defining~$\alpha$.
If the endpoint of an arc corresponds to a vertex of the graph, we also call the endpoint an \textit{endvertex} of the arc.
Similarly as for paths, we call an arc an $x$--$y$ arc if $x$ and $y$ are the endpoints of the arc.
Possibly the simplest example of a nontrivial arc is a ray together with the end it converges to.
However, the structure of arcs is more complicated in general and they might contain up to $2^{\aleph_0}$ many ends.
We call a subspace $X$ of $|G|$ \textit{arc-connected} if for any two points $x$ and $y$ of $X$ there is an $x$--$y$ arc in $X$.

Using the notions of circles and arc-connectedness we now extend trees in a similar topological way.
We call an arc-connected standard subspace of $|G|$ a \textit{topological tree} if it does not contain any circle.
Note that, similar as for finite trees, for any two points $x, y$ of a topological tree there is a unique $x$--$y$ arc in that topological tree.
Generalizing the definition of caterpillars, we call a topological tree $\overline{T}$ in $|G|$ a \textit{topological caterpillar} if $\overline{T-L}$ is an arc, where $T$ is a forest in $G$ and $L$ denotes the set of all leaves of $T$, i.e., vertices of degree $1$ in~$T$.

Now let $\omega$ be an end of a locally finite connected graph $G$.
We define the \textit{vertex-} or \textit{edge-degree of $\omega$ in $G$} as the supremum of the number of vertex- or edge-disjoint rays in $G$, respectively, which are contained in $\omega$.
By this definition ends may have infinite vertex- or edge-degree.
Similarly, we define the \textit{vertex-} or \textit{edge-degree of $\omega$ in a standard subspace $X$ of $|G|$} as the supremum of vertex- or edge-disjoint arcs in $X$, respectively, that have $\omega$ as an endpoint.
We should mention here that the supremum is actually an attained maximum in both definitions.
Furthermore, when we consider the whole space $|G|$ as a standard subspace of itself, the vertex-degree in $G$ of any end $\omega$ of $G$ coincides with the vertex-degree in $|G|$ of $\omega$.
The same holds for the edge-degree.
The proofs of these statements are nontrivial and since it is enough for us to work with the supremum, we will not go into detail here.

We make one last definition with respect to end degrees which allows us to distinguish the parity of degrees of ends when they are infinite.
The idea of this definition is due to Bruhn and Stein~\cite{circle}.
We call the vertex- or edge-degree of an end~$\omega$ of $G$ in a standard subspace $X$ of $|G|$ \textit{even} if there is a finite set $S \subseteq V(G)$ such that for every finite set $S' \subseteq V(G)$ with $S \subseteq S'$ the maximum number of vertex- or edge-disjoint arcs in $X$, respectively, with $\omega$ as endpoint and some $s \in S'$ is even. Otherwise, we call the vertex- or edge-degree of $\omega$ in~$X$, respectively,~\textit{odd}.

Next we collect some useful statements about the space $|G|$ for a locally finite connected graph $G$.

\begin{proposition}\label{compact}\cite[Prop.\ 8.5.1]{diestel_buch}
If $G$ is a locally finite connected graph, then $|G|$ is a compact Hausdorff space.
\end{proposition}

Having Proposition~\ref{compact} in mind the following basic lemma helps us to work with continuous maps and to verify homeomorphisms, for example when considering circles or arcs.

\begin{lemma}\label{invers_cont}
Let $X$ be a compact space, $Y$ be a Hausdorff space and $f: X \longrightarrow Y$ be a continuous injection. Then $f^{-1}$ is continuous too.
\end{lemma}

The following lemma tells us an important combinatorial property of arcs.
To state the lemma more easily, let $\mathring{F}$ denote the set of inner points of edges $e \in F$ in $|G|$ for an edge set $F \subseteq E(G)$.

\begin{lemma}\label{jumping-arc}\cite[Lemma 8.5.3]{diestel_buch}
Let $G$ be a locally finite connected graph and \linebreak ${F \subseteq E(G)}$ be a cut with sides $V_1$ and $V_2$.
\textnormal{
\begin{enumerate}[\normalfont(i)]
\item \textit{If $F$ is finite, then $\overline{V_1} \cap \overline{V_2} = \emptyset$, and there is no arc in $|G| \setminus \mathring{F}$ with one endpoint in $V_1$ and the other in $V_2$.}
\item \textit{If $F$ is infinite, then $\overline{V_1} \cap \overline{V_2} \neq \emptyset$, and there may be such an arc.}
\end{enumerate}
}
\end{lemma}

The next lemma ensures that connectedness and arc-connectedness are equivalent for the spaces we are mostly interested in, namely standard subspaces, which are closed by definition.

\begin{lemma}\label{arc_conn}\cite[Thm.\ 2.6]{path-cyc-tree}
If $G$ is a locally finite connected graph, then every closed topologically connected subset of $|G|$ is arc-connected.
\end{lemma}

We continue in the spirit of Lemma~\ref{jumping-arc} by characterising important topological properties of the space $|G|$ in terms of combinatorial ones.
The following lemma deals with arc-connected subspaces.
It will be convenient for us to use this in a proof later on.

\begin{lemma}\label{top_conn}\cite[Lemma 8.5.5]{diestel_buch}
If $G$ is a locally finite connected graph, then a standard subspace of $|G|$ is topologically connected (equivalently: arc-connected) if and only if it contains an edge from every finite cut of $G$ of which it meets both sides.
\end{lemma}

The next theorem is actually part of a bigger one containing more equivalent statements.
Since we shall need only one equivalence, we reduced it to the following formulation.
For us it will be helpful to check or at least bound the degree of an end in a standard subspace just by looking at finite cuts instead of dealing with the homeomorphisms that actually define the relevant arcs.

\begin{theorem}\label{cycspace}\cite[Thm.\ 2.5]{diestel_arx}
Let $G$ be a locally finite connected graph. Then the following are equivalent for $D \subseteq E(G)$:
\begin{enumerate}[\normalfont(i)]
\item $D$ meets every finite cut in an even number of edges.
\item Every vertex of $G$ has even degree in $\overline{D}$ and every end of $G$ has even edge-degree in $\overline{D}$.
\end{enumerate}
\end{theorem}

The following lemma gives us a nice combinatorial description of circles and will be especially useful in combination with Theorem~\ref{cycspace} and Lemma~\ref{top_conn}.

\begin{lemma}\label{circ}\cite[Prop.\ 3]{circle}
Let $C$ be a subgraph of a locally finite connected graph~$G$. Then $\overline{C}$ is a circle if and only if $\overline{C}$ is topologically connected, every vertex in $\overline{C}$ has degree $2$ in $\overline{C}$ and every end of $G$ contained in $\overline{C}$ has edge-degree $2$ in $\overline{C}$.
\end{lemma}

A basic fact about finite Hamiltonian graphs is that they are always $2$-connected.
For locally finite connected graphs this is also a well-known fact, although it has not separately been published.
Since we shall need this fact later and can easily deduce it from the lemmas above, we include a proof here.

\begin{corollary}\label{2-con}
Every locally finite connected Hamiltonian graph is $2$-connected.
\end{corollary}

\begin{proof}
Let $G$ be a locally finite connected Hamiltonian graph and suppose for a contradiction that it is not $2$-connected.
Fix a subgraph $C$ of $G$ whose closure $\overline{C}$ is a Hamilton circle of $G$ and a cut vertex $v$ of $G$.
Let $K_1$ and $K_2$ be two different components of $G-v$.
By Theorem~\ref{cycspace} the circle $\overline{C}$ uses evenly many edges of each of the finite cuts $\delta(K_1)$ and $\delta(K_2)$.
Since $\overline{C}$ is a Hamilton circle and, therefore, topologically connected, we also get that it uses at least two edges of each of these cuts by Lemma~\ref{top_conn}.
This implies that $v$ has degree at least $4$ in $C$, which contradicts Lemma~\ref{circ}.
\end{proof}

\section{Topological caterpillars}

In this section we close a gap with respect to the general question of when the $k$-th power of a graph has a Hamilton circle.
Let us begin by summarizing the results in this field.
We start with finite graphs.
The first result to mention is the famous theorem of Fleischner, Theorem~\ref{fin_fleisch}, which deals with $2$-connected graphs.

For higher powers of graphs the following theorem captures the whole situation.

\begin{theorem}\label{fin_3-con}\cite{karaganis, sekanina}
The cube of any finite connected graph on at least three vertices is Hamiltonian.
\end{theorem}

These theorems leave the question whether and when one can weaken the assumption of being $2$-connected and still maintain the property of being Hamiltonian.
Theorem~\ref{fin_cater} gives an answer to this question.

Now let us turn our attention towards locally finite infinite graphs.
As mentioned in the introduction, Georgakopoulos has completely generalized Theorem~\ref{fin_fleisch} to locally finite graphs by proving Theorem~\ref{inf-fleisch}.
Furthermore, he also gave a complete generalization of Theorem~\ref{fin_3-con} to locally finite graphs with the following theorem.

\begin{theorem}\label{inf-3-con}\cite[Thm.\ 5]{agelos}
The cube of any locally finite connected graph on at least three vertices is Hamiltonian.
\end{theorem}

What is left and what we do in the rest of this section is to prove lemmas about locally finite graphs covering implications similar to those in Theorem~\ref{fin_cater}, and mainly Theorem~\ref{top_catp_HC}, which extends Corollary~\ref{fin_cater_impl} to locally finite graphs.

Let us first consider a naive way of extending Theorem~\ref{fin_cater} and Corollary~\ref{fin_cater_impl} to locally finite graphs.
Since we consider spanning caterpillars for Corollary~\ref{fin_cater_impl}, we need a definition of these objects in infinite graphs that allows them to contain infinitely many vertices.
So let us modify the definition of caterpillars as follows:
A locally finite tree is called a \textit{caterpillar} if after deleting its leaves only a finite path, a ray or a double ray is left.
Using this definition Theorem~\ref{fin_cater} remains true for locally finite infinite trees $T$ and Hamilton circles in $|T^2|$.
The same proof as the one Harary and Schwenk~\cite[Thm.\ 1]{cat_HC} gave for Theorem~\ref{fin_cater} in finite graphs can be used to show this.

Corollary~\ref{fin_cater_impl} remains also true for locally finite graphs using this adapted definition of caterpillars.
Its proof, however, is no trivial deduction anymore.
The problem is that for a spanning tree $T$ of a locally finite connected graph $G$ the topological spaces $|T^2|$ and $|G^2|$ might differ not only in inner points of edges but also in ends.
More precisely, there might be two equivalent rays in $G^2$ that belong to different ends of $T^2$.
So the Hamiltonicity of $T^2$ does not directly imply the one of $G^2$.
However, for $T$ being a spanning caterpillar of $G$, this problem can only occur when $T$ contains a double ray such that all subrays belong to the same end of $G$.
Then the same construction as in the proof for the implication from~(iii) to (i) of Theorem~\ref{fin_cater} can be used to build a spanning double ray in $T^2$ which is also a Hamilton circle in $|G^2|$.
The idea for the construction which is used for this implication is covered in Lemma~\ref{order}.

The downside of this naive extension is the following.
For a locally finite infinite graph the assumption of having a spanning caterpillar is quite restrictive.
Such graphs can especially have at most two ends since having three ends would imply that the spanning caterpillar must contain three disjoint rays.
This, however, is impossible because it would force the caterpillar to contain a $S(K_{1, 3})$.
For this reason we have defined a topological version of caterpillars, namely topological caterpillars.
Their definition allows graphs with arbitrary many ends to have a spanning topological caterpillar.
Furthermore, it yields with Theorem~\ref{top_catp_HC} a more relevant extension of Corollary~\ref{fin_cater_impl} to locally finite graphs.

We briefly recall the definition of topological caterpillars.
Let $G$ be a locally finite connected graph.
A topological tree $\overline{T}$ in $|G|$ is a topological caterpillar if $\overline{T-L}$ is an arc, where $T$ is a forest in $G$ and $L$ denotes the set of all leaves of $T$, i.e., vertices of degree $1$ in $T$.

The following basic lemma about topological caterpillars is easy to show and so we omit its proof.
It is an analogue of the equivalence of the statements~(ii) and~(iii) of Theorem~\ref{fin_cater} for topological caterpillars.

\begin{lemma}\label{_inf_forbd_sub}
Let $G$ be a locally finite connected graph.
A topological tree $\overline{T}$ in $|G|$ is a topological caterpillar if and only if $T$ does not contain $S(K_{1,3})$ as a subgraph and all ends of $G$ have vertex-degree in $\overline{T}$ at most $2$.
\end{lemma}

Before we completely turn towards the preparation of the proof of Theorem~\ref{top_catp_HC} let us consider statement~(i) of Theorem~\ref{fin_cater} again.
A complete extension of Theorem~\ref{fin_cater} to locally finite graphs using topological caterpillars seems impossible because of statement~(i).
To see this we should first make precise what the adapted version of statement~(i) most possibly should be.
In order to state it let $G$ denote a locally finite connected graph and let $\overline{T}$ be a topological tree in $|G|$.
Now the formulation of the adapted statement should be as follows:
\vspace{6pt}
\textnormal{
\begin{enumerate}[\normalfont(i*)]
\item \textit{In the subspace $\overline{T^2}$ of $|G^2|$ is a circle containing all vertices of $T$.}
\end{enumerate}
}

\vspace{6pt}

\noindent This statement does not hold if $T$ has more than one graph theoretical component.
Therefore, it cannot be equivalent to $\overline{T}$ being a topological caterpillar in $|G|$, which is the adapted version of statement~(iii) of Theorem~\ref{fin_cater} for locally finite graphs.
Note that any two vertices of $T$ lie in the same graph theoretical component of $T$ if and only if they lie in the same graph theoretical component of~$T^2$.
Hence, we can deduce that statement~(i*) fails if $T$ has more than one graph theoretical component from the following claim.

\begin{claim}
Let $G$ be a locally finite connected graph and let $\overline{T}$ be a topological tree in~$|G|$.
Then there is no circle in the subspace $\overline{T^2}$ of $|G^2|$ that contains vertices from different graph theoretical components of $T^2$.
\end{claim}

\begin{proof}
We begin with a basic observation.
The inclusion map from $G$ into $G^2$ induces an embedding from $|G|$ into $|G^2|$ in a canonical way.
Moreover, all ends of $G^2$ are contained in the image of this embedding.
To see this note that any two non-equivalent rays in $G$ stay non-equivalent in $G^2$ since $G$ is locally finite.
Furthermore, by applying Lemma~\ref{star-comb} it is easy to see that every end in $G^2$ contains a ray that is also a ray of $G$.
This already yields an injection from $|G|$ to $|G^2|$ whose image contains all of $\Omega(G^2)$.
Verifying the continuity of this map and its inverse is immediate.

Now let us suppose for a contradiction that there is a circle $C$ in $\overline{T^2}$ containing vertices $v, v'$ from two different graph theoretical components $K, K'$ of $T^2$.
Say $v \in V(K)$ and $v' \in V(K')$.
Let $A_1$ and $A_2$ denote the two $v'$--$v$ arcs on $C$.
Since $A_1$ and $A_2$ are disjoint except from their endpoints, they have to enter $K$ via different ends $\omega^2_1$ and $\omega^2_2$ of $G^2$ that are contained in $\overline{K} \subseteq |G^2|$.
Say $\omega^2_1 \in A_1$ and $\omega^2_2 \in A_2$.
By the observation above $\omega^2_1$ and $\omega^2_2$ correspond to two different ends $\omega_1$ and $\omega_2$ of~$G$.
Only one of them, say~$\omega_1$, lies on the unique $v'$--$v$ arc that is contained in the topological tree $\overline{T}$.
Now we modify $A_2$ by replacing each edge $uw$ of $A_2$ which is not in $E(T)$ by a $u$--$w$ path of length $2$ that lies in $T$.
By Lemma~\ref{top_conn} this yields an arc-connected subspace of $\overline{T}$ that contains $v$ and $v'$.
By our observation above the unique $v'$--$v$ arc in this subspace must contain the end $\omega_2$.
This, however, is a contradiction since we have found two different $v'$--$v$ arcs in $\overline{T}$.
\end{proof}

Now we start preparing the proof of Theorem~\ref{top_catp_HC}.
For this we define a certain partition of the vertex set of a topological caterpillar.
Additionally, we define a linear order of these partition classes.
Let $G$ be a locally finite connected graph and $\overline{T}$ a topological caterpillar in~$|G|$.
Furthermore, let $L$ denote the set of leaves of $T$.
By definition, $\overline{T-L}$ is an arc, call it $A$.
This arc induces a linear order $<_A$ of the vertices of $V(T)-L$.
For consecutive vertices $v, w \in V(T)-L$ with $v <_A w$ we now define the set
\[{P_w := \lbrace w \rbrace \cup (N_T(v) \cap L)} \]
(cf.~Figure~\ref{partition}).
If $A$ has a maximal element $m$ with respect to $<_A$, we define an additional set $P^+ = N_T(m) \cap L$.
Should $A$ have a minimal element $s$ with respect to $<_A$, we define another additional set $P^- = \lbrace s \rbrace$.
The sets $P_w$, possibly together with $P^+$ and $P^-$, form a partition $\mathcal{P}_T$ of $V(T)$.
For any $v \in V(T)$ we denote the corresponding partition class containing $v$ by $V_v$.
Next we use the linear order $<_A$ to define a linear order $<_T$ on $\mathcal{P}_T$.
For any two vertices ${v, w \in V(T)-L}$ with $v <_A w$ set $V_v <_T V_w$.
If $P^+$ (resp.~$P^-$) exists, set $P_v <_T P^+$ (resp.~${P^- <_T P_v}$) for every ${v \in V(T)-L}$.
Finally we define for two vertices~${v, w \in V(T)}$ with $V_v \leq_T V_w$ the set
\[I_{vw} := \bigcup \lbrace V_u \; ; \; V_v \leq_T V_u \leq_T V_w \rbrace. \]

\begin{figure}[htbp]
\centering
\includegraphics{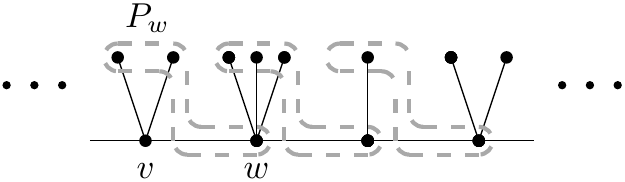}
\caption{The partition classes $P_w$.}
\label{partition}
\end{figure}

The following basic lemma lists important properties of the partition $\mathcal{P}_T$ together with its order $<_T$.
The proof of this lemma is immediate from the definitions of $\mathcal{P}_T$ and $<_T$.
Especially for Lemma~\ref{decomp} and in the proof of Theorem~\ref{top_catp_HC} the listed properties will be applied intensively.
Furthermore, the proof that statement~(iii) of Theorem~\ref{fin_cater} implies statement~(i) of Theorem~\ref{fin_cater} follows easily from this lemma.

\begin{lemma}\label{order}
Let $\overline{T}$ be a topological caterpillar in $|G|$ for a locally finite connected graph~$G$. Then the partition $\mathcal{P}_T$ of $V(T)$ has the following properties:
\begin{enumerate}[\normalfont(i)]
\item Any two different vertices belonging to the same partition class of $\mathcal{P}_T$ have distance~$2$ from each other in $T$.
\item For consecutive partition classes $Q$ and $R$ with $Q <_T R$, there is a unique vertex in $Q$ that has distance $1$ in $T$ to every vertex of $R$.
For $Q \neq P^-$, this vertex is the one of $Q$ that is not a leaf of $T$.
\end{enumerate}
\end{lemma}

\begin{proof}
\end{proof}

Referring to statement~(ii) of Lemma~\ref{order}, let us call the vertex in a partition class~$Q \in \mathcal{P}_T$ that is not a leaf of $T$ the \textit{jumping vertex} of $Q$.

We still need a bit of notation and preparation work before we can prove the main theorem of this section.
Now let $\overline{T}$ denote a topological caterpillar with only one graph-theoretical component.
Let $(\mathcal{X}_1, \mathcal{X}_2)$ be a bipartition of the partition classes $V_v$ such that consecutive classes with respect to $\leq_T$ lie not both in $\mathcal{X}_1$, or in $\mathcal{X}_2$.
Furthermore, let $v, w \in V(T)$ be two vertices, say with $V_v \leq_T V_w$, whose distance is even in $T$.
We define a $(v, w)$ \textit{square string} $S$ in $T^2$ to be a path in $T^2$ with the following properties:

\begin{enumerate}
\item $S$ uses only vertices of partitions that lie in the bipartition class $\mathcal{X}_i$ in which $V_v$ and $V_w$ lie.
\item $S$ contains all vertices of partition classes $V_u \in \mathcal{X}_i$ for $V_v <_T V_u <_T V_w$.
\item $S$ contains only $v$ and $w$ from $V_v$ and $V_w$, respectively.
\end{enumerate}
Similarly, we define $(v, w]$, $[v, w)$ and $[v, w]$ square strings in $T^2$, but with the difference in (3) that they shall also contain all vertices of $V_w$, $V_v$ and $V_v \cup V_w$, respectively.
We call the first two types of square strings \textit{left open} and the latter ones \textit{left closed}.
The notion of being \textit{right open} and \textit{right closed} is analogously defined.
From the properties of $\mathcal{P}_T$ listed in Lemma~\ref{order}, it is immediate how to construct square strings.

The next lemma gives us two possibilities to cover the vertex set of a graph-theoretical component of a topological caterpillar $\overline{T}$ that contains a double ray.
Each cover will consist of two, possibly infinite, paths of $T^2$.
Later on we will use these covers to connect all graph-theoretical components of $\overline{T}$ in a certain way such that a Hamilton circle of $G^2$ is formed.

\begin{lemma}\label{decomp}
Let $G$ be a locally finite connected graph and let $\overline{T}$ be a topological caterpillar in $|G|$.
Suppose $T$ has only one graph-theoretical component and contains a double ray.
Furthermore, let $v$ and $w$ be vertices of $T$ with $V_v \leq_T V_w$.
\begin{enumerate}[\normalfont(i)]

\item If $d_T(v, w)$ is even, then in $T^2$ there exist a $v$--$w$ path $P$, a double ray $D$ and two rays $R_v$ and $R_w$ with the following properties:

\begin{itemize}
\item $P$ and $D$ are disjoint as well as $R_v$ and $R_w$.
\item $V(T) = V(P) \cup V(D) = V(R_v) \cup V(R_w)$.
\item $v$ and $w$ are the start vertices of $R_v$ and $R_w$, respectively.
\item $R_v \cap V_x = \emptyset$ for every $V_x >_T V_w$.
\item $R_w \cap V_y = \emptyset$ for every $V_y <_T V_v$.
\end{itemize}

\vspace{5 pt}

\item If $d_T(v, w)$ is odd, then in $T^2$ there exist rays $R_v, R_w, R'_v, R'_w$ with the following properties:

\begin{itemize}
\item $R_v$ and $R_w$ are disjoint as well as $R'_v$ and $R'_w$.
\item $V(T) = V(R_v) \cup V(R_w) = V(R'_v) \cup V(R'_w)$.
\item $v$ is the start vertex of $R_v$ and $R'_v$ while $w$ is the one of $R_w$ and $R'_w$.
\item $R_v \cap V_x = R'_w \cap V_x = \emptyset$ for every $V_x >_T V_w$.
\item $R_w \cap V_y = R'_v \cap V_y = \emptyset$ for every $V_y <_T V_v$.
\end{itemize}
\end{enumerate}
\end{lemma}

\begin{proof}
We sketch the proof of statement (i).
As $v$--$w$ path $P$ we take a square string $S_{vw}$ in $T^2$ with $v$ and $w$ as endvertices.
Depending whether $v$ is a jumping vertex or not we take a left open or closed square string, respectively.
Depending on $w$ we take a right closed or open square string if $w$ is a jumping vertex or not, respectively.
Since $d_T(v, w)$ is even, we can find such square strings.
To construct the double ray $D$ start with a $(v^-, w^-]$ square string in $T^2$ where $v^-$ and $w^-$ denote the jumping vertices in the partition classes proceeding $V_v$ and $V_w$, respectively.
Using the properties~(i) and (ii) of the partition $\mathcal{P}_T$ mentioned in Lemma~\ref{order}, the $(v^-, w^-]$ square string can be extend to a desired double ray $D$ containing all vertices of $T$ that do not lie in $S_{vw}$ (cf.~Figure~\ref{caterpillar_pattern}).

To define $R_v$ we start with a square string $S_v$ having $v$ as one endvertex.
For the definition of $S_v$ we distinguish four cases.
If $v$ and $w$ are jumping vertices, we set $S_v$ as a path obtained by taking a $(v, w]$ square string and deleting $w$ from it.
If $v$ is not a jumping vertex, but~$w$ is one, take a $[v, w]$ square string, delete $w$ from it and set the remaining path as~$S_v$.
In the case that $v$ is a jumping vertex, but $w$ is none, $S_v$ is defined as a path obtained from a $(v, w)$ square string from which we delete $w$.
In the case that neither $v$ nor $w$ is a jumping vertex, we take a $[v, w)$ square string, delete $w$ from it and set the remaining path as $S_v$.
Next we extend $S_v$ using a square string to a path with $v$ as one endvertex containing all vertices in partition classes $V_u$ with $V_v <_T V_u <_T V_w$.
We extend the remaining path to a ray that contains also all vertices in partition classes $V_u$ with $V_u \leq_T V_v$, but none from partition classes $V_x$ for $V_x >_T V_w$.
The desired second ray $R_w$ can now easily be build in $T^2 -R_v$.

The rays for statement (ii) are defined in a very similar way (cf.~Figure~\ref{caterpillar_pattern}).
Therefore, we omit their definitions here.
\end{proof}

\begin{figure}[htbp]
\centering
\includegraphics{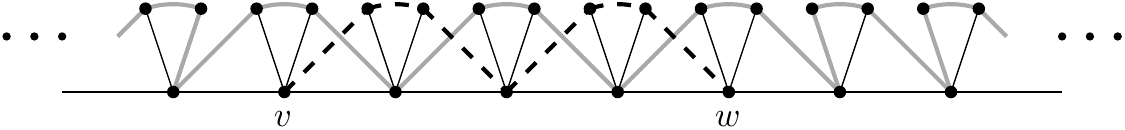} \vspace{10pt}

\includegraphics{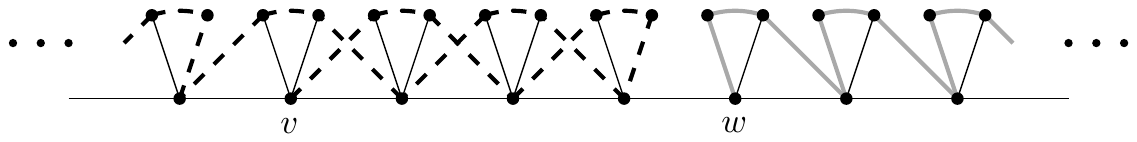} \vspace{10pt}

\includegraphics{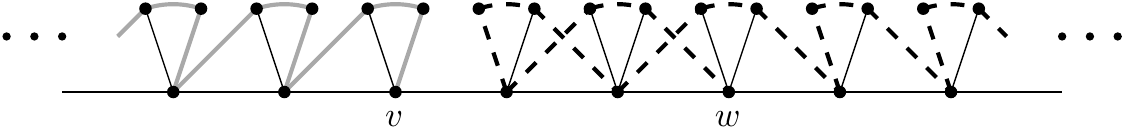} \vspace{10pt}

\includegraphics{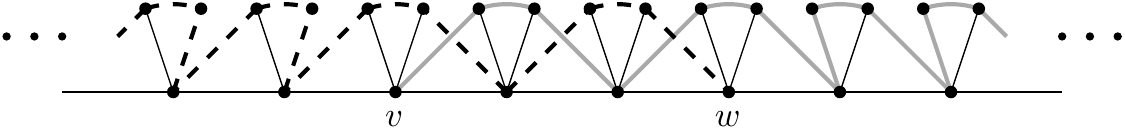} 
\caption{Examples for covering the vertices of a caterpillar as in Lemma~\ref{decomp}.}
\label{caterpillar_pattern}
\end{figure}

The following lemma is essential for connecting the parts of the vertex covers of two different graph-theoretical components of $\overline{T}$.
Especially, here we make use of the structure of $|G|$ instead of arguing only inside of $\overline{T}$ or $\overline{T^2}$.
This allows us to build a Hamilton circle using square strings and to ``jump over" an end to avoid producing an edge-degree bigger than $2$ at that end.

\begin{lemma}\label{shortcut}
Let $\overline{T}$ be a spanning topological caterpillar of a locally finite connected graph $G$ and let $v, w \in V(G)$ where $V_v \leq_T V_w$.
Then for any two vertices $x, y$ with ${V_v <_T V_x <_T V_w}$ and ${V_v <_T V_y <_T V_w}$ there exists a finite $x$--$y$ path in $G[I_{vw}]$.
\end{lemma}

\begin{proof}
Let the vertices $v, w, x$ and $y$ be as in the statement of the lemma and, as before, let $L$ denote the set of leaves of $T$.
Now suppose for a contradiction that there is no finite $x$--$y$ path in $G[I_{vw}]$.
Then we can find an empty cut $D$ of $G[I_{vw}]$ with sides $M$ and $N$ such that $x$ and $y$ lie on different sides of it.
Since $\overline{T \cap G[I_{vw}]}$ contains an $x$--$y$ arc, there must exist an end $\omega \in \overline{M} \cap \overline{N} \cap \overline{T-L}$.

Let us show next that there exists an open set $O$ in $|G|$ that contains $\omega$ and, additionally, every vertex in $O$ is an element of $I_{vw}$.
To see this we first pick a set $O_A \subseteq \overline{T-L}$ so that it is open in the subspace $\overline{T-L}$, topologically connected and contains $\omega$, but its closure does not contain the jumping vertices of $V_v$ and~$V_w$.
Now let $O'$ be an open set in $|G|$ witnessing that $O_A$ is open in $\overline{T-L}$.
We prove that $O'$ contains only finitely many vertices of $V(G) \setminus I_{vw}$.
Suppose for a contradiction that this is not the case.
Then we would find an infinite sequence $(z_n)_{n \in \mathbb{N}}$ of different vertices in $O' \setminus I_{vw}$ that must converge to some point $p \in |G|$ by the compactness of $|G|$.
Since~$\overline{T}$ is a spanning topological caterpillar of $G$, it contains all the vertices~$z_n$.
Using that $G$ is locally finite, we get that the jumping vertices of the sets $V_{z_n}$ also form a sequence that converges to $p$.
So we can deduce that ${p \in \overline{T-L}}$, because $\overline{T-L}$ is a closed subspace containing all jumping vertices.
Hence, ${p \in \overline{O'} \cap (\overline{T-L}) = \overline{O_A}}$.
This is a contradiction to our choice of $O_A$ ensuring $p \notin \overline{O_A}$.
Hence, $O'$ contains only finitely many vertices of $V(G) \setminus I_{vw}$, say $v_1, \ldots, v_n$ for some $n \in \mathbb{N}$.
Before we define our desired set $O$ using $O'$, note that ${O_v := |G| \setminus \lbrace v \rbrace}$ defines an open set in $|G|$ for every vertex $v \in V(G)$.
Therefore, ${O := O' \cap \bigcap^n_{i=1} O_{v_i}}$ is an open set in $|G|$ containing no vertex of $V(G) \setminus I_{vw}$.

Inside $O$ we can find a basic open set $B$ around $\omega$, which contains a graph-theoretical connected subgraph with all vertices of $B$.
Now $B$ contains vertices of $M$ and $N$ as well as a finite path between them, which must then also exist in $G[I_{vw}]$.
Such a path would have to cross $D$ contradicting the assumption that $D$ is an empty cut in $G[I_{vw}]$.
\end{proof}

To figure out which parts of the vertex covers of which graph-theoretical components of~$\overline{T}$ we can connect such that afterwards we are still able to extend this construction to a Hamilton circle of $G$, we shall use the next lemma.
For the formulation of the lemma, we use the notion of \textit{splits}.

Let $G$ be a multigraph and $v \in V(G)$.
Furthermore, let $E_1, E_2 \subseteq \delta(v)$ such that $E_1 \cup E_2 = \delta(v)$ but $E_1 \cap E_2 = \emptyset$ where $E_i \neq \emptyset$ for $i \in \lbrace 1, 2 \rbrace$.
Now we call a multigraph $G'$ a $v$\textit{-split} of $G$ if
\[{V(G') = V(G) \setminus \lbrace v \rbrace \cup \lbrace v_1, v_2 \rbrace}\]
with ${v_1, v_2 \notin V(G)}$ and
\[{E(G') = E(G-v) \cup \lbrace v_1w \; ; \; wv \in E_1 \rbrace \cup \lbrace v_2u \; ; \; uv \in E_2 \rbrace}.\]
We call the vertices $v_1$ and $v_2$ \textit{replacement vertices} of $v$.

\begin{lemma}\label{euler_split}
Let $G$ be a finite Eulerian multigraph and $v$ be a vertex of degree $4$ in~$G$.
Then there exist two $v$-splits $G_1$ and $G_2$ of $G$ both of which are also Eulerian.
\end{lemma}

\begin{proof}
There are ${\frac{1}{2} } \cdot {\binom {4} {2}} = 3$ possible non-isomorphic $v$-splits of $G$ such that $v_1$ and $v_2$ have degree $2$ in the $v$-split.
Assume that one of them, call it $G'$, is not Eulerian.
This can only be the case if $G'$ is not connected.
Let $(A, B)$ be an empty cut of $G'$.
Note that $G-v$ has precisely two components $C_1$ and $C_2$ since $G$ is Eulerian and $v$ has degree $4$ in $G$.
So $C_1$ and $C_2$ must lie in different sides of $(A, B)$, say $C_1 \subseteq A$.
Since $G$ was connected, we get that $v_1$ and $v_2$ lie in different sides of the cut $(A, B)$, say $v_1 \in A$.
Therefore, ${A = C_1 \cup \lbrace v_1 \rbrace}$ and ${B = C_2 \cup \lbrace v_2 \rbrace}$.
If ${\delta(v) = \lbrace vw_1, vw_2, vw_3, vw_4 \rbrace}$ and ${\lbrace v_1w_1, v_1w_2 \rbrace , \lbrace v_2w_3, v_2w_4 \rbrace \subseteq E(G')}$, set $G_1$ and $G_2$ as $v$-splits of $G$ such that the inclusions $\lbrace v_1w_1, v_1w_3 \rbrace , \lbrace v_2w_2, v_2w_4 \rbrace \subseteq E(G_1)$ and $\lbrace v_1w_1, v_1w_4 \rbrace , \lbrace v_2w_2, v_2w_3 \rbrace \subseteq E(G_2)$ hold.
Now $G_1$ and $G_2$ are Eulerian, because every vertex has even degree in each of those multigraphs and both multigraphs are connected.
To see the latter statement, note that any empty cut $(X, Y)$ of $G_i$ for $i \in \lbrace 1, 2 \rbrace$ would need to have $C_1$ and $C_2$ on different sides.
If also $v_1$ and $v_2$ are on different sides, we would have $(A, B) = (X, Y)$, which does not define an empty cut of $G_i$ by definition of $G_i$.
However, $v_1$ and $v_2$ cannot lie on the same side of the cut $(X, Y)$.
This is because otherwise the cut $(X, Y)$ would induce an empty cut in $G$ after identifying $v_1$ and $v_2$ in $G_i$.
Since $G$ is Eulerian and therefore especially connected, we would have a contradiction.
\end{proof}

Now we have all tools together to prove Theorem~\ref{top_catp_HC}.
Before we start the proof, let us recall the statement of the theorem.

\setcounter{section}{1}
\setcounter{theorem}{4}
\begin{theorem}
The square of any locally finite connected graph $G$ on at least three vertices such that $|G|$ contains a spanning topological caterpillar is Hamiltonian.
\end{theorem}
\setcounter{theorem}{7}
\setcounter{section}{3}

\begin{proof}
Let $G$ be a graph as in the statement of the theorem and let $\overline{T}$ be a spanning topological caterpillar of $G$.
We may assume by Corollary~\ref{fin_cater_impl} that $G$ has infinitely many vertices.
Now let us fix an enumeration of the vertices, which is possible since every locally finite connected graph is countable.
We inductively build a Hamilton circle of $G^2$ in at most $\omega$ many steps.
We ensure that in each step $i \in \mathbb{N}$ we have two disjoint arcs $\overline{A^i}$ and $\overline{B^i}$ in $|G^2|$ whose endpoints are vertices of subgraphs $A^i$ and $B^i$ of $G^2$, respectively.
Let $a^i_{\ell}$ and $a^i_{r}$ (resp.\ $b^i_{\ell}$ and $b^i_{r}$) denote the endvertices of $\overline{A^i}$ (resp.\ $\overline{B^i}$) such that $V_{a^i_{\ell}} \leq_T V_{a^i_{r}}$ (resp.\ $V_{b^i_{\ell}} \leq_T V_{b^i_{r}}$).
For the construction we further ensure the following properties in each step $i \in \mathbb{N}$:

\begin{enumerate}
\item The vertices $a^i_{r}$ and $b^i_{r}$ are the jumping vertices of $V_{a^i_{r}}$ and $V_{b^i_{r}}$, respectively.
\item The partition sets $V_{a^i_{\ell}}$ and $V_{b^i_{\ell}}$ as well as $V_{a^i_{r}}$ and $V_{b^i_{r}}$ are consecutive with respect to $\leq_T$.
\item If $V_v \cap V(A^i \cup B^i) \neq \emptyset$ holds for any vertex $v \in V(G)$, then $V_v \subseteq V(A^i \cup B^i)$.
\item If for any vertex $v \in V(G)$ there are vertices $u, w \in V(G)$ such that ${V_u, V_w \subseteq V(A^i \cup B^i)}$ and $V_u \leq_T V_v \leq_T V_w$, then $V_v \subseteq V(A^i \cup B^i)$ is true.
\item $A^i \cap A^{i+1} = A^i$ and $B^i \cap B^{i+1} = B^i$, but $V(A^{i+1} \cup B^{i+1})$ contains the least vertex with respect to the fixed vertex enumeration that was not already contained in $V(A^i \cup B^i)$.
\end{enumerate}

We start the construction by picking two adjacent vertices $t$ and $t'$ in $T$ that are no leaves in $T$.
Then $V_t$ and $V_t'$ are consecutive with respect to $\leq_T$.
Note that $G^2[V_t]$ and $G^2[V_{t'}]$ are cliques by property~(i) of the partition $\mathcal{P}_T$ mentioned in Lemma~\ref{order}.
We set $A^1$ to be a Hamilton path of $G^2[V_t]$ with endvertex $t$ and $B^1$ to be one of $G^2[V_{t'}]$ with endvertex $t'$.
This completes the first step of the construction.

Suppose we have already constructed $A^n$ and $B^n$.
Let $v \in V(G)$ be the least vertex with respect to the fixed vertex enumeration that is not already contained in $V(A^n \cup B^n)$.
We know by our construction that either $V_v <_T V_x$ or $V_v >_T V_x$ for every vertex $x \in V(A^n \cup B^n)$.
Consider the second case, since the argument for the first works analogously.
Let $v' \in V(G)$ be a vertex such that $V_{v'}$ is the predecessor of $V_{v}$ with respect to $\leq_T$.
Further, let $w \in V(G)$ be a vertex such that $V_w >_T V_{a^n_{r}}, V_{b^n_{r}}$ and $V_w$ is the successor of either $V_{a^n_{r}}$ or $V_{b^n_{r}}$, say $V_{b^n_{r}}$.
By Lemma~\ref{shortcut} there exists a $v'$--$w$ path $P$ in $G[I_{b^n_r, v}]$.
We may assume that $E(P) \setminus E(T)$ does not contain an edge whose endvertices lie in the same graph-theoretical component of~$T$.
Furthermore, we may assume that every graph-theoretical component of $T$ is incident with at most two edges of~$E(P) \setminus E(T)$.
Otherwise we could modify the path $P$ using edges of $E(T)$ to meet these conditions.

Next we inductively define a finite sequence of finite Eulerian auxiliary multigraphs $H_1, \ldots, H_k$ where $H_k$ is a cycle for some $k \in \mathbb{N}$.
Every vertex in each of these multigraphs will have either degree $2$ or degree $4$.
Furthermore, we shall obtain $H_{i+1}$ from $H_i$ as a $h$-split for some vertex $h \in V(H_i)$ of degree $4$ until we end up with a multigraph $H_k$ that is a cycle.

As $V(H_1)$ take the set of all graph-theoretical components $T_1, \ldots, T_n$  of $T$ that are incident with an edge of $E(P) \setminus E(T)$.
Two vertices $T_i$ and $T_j$ are adjacent if either there is an edge in $E(P) \setminus E(T)$ whose endpoints lie in $T_i$ and $T_j$ or there is a $t_i$--$t_j$ arc $\overline{A}$ in $\overline{T}$ for a subgraph $A$ of $T$ and vertices $t_i \in V(T_i)$ and $t_j \in V(T_j)$ such that no endvertex of any edge of $E(P) \setminus E(T)$ lies in $V(A) \cup N_T(A)$.
Since $\overline{T}$ is a spanning topological caterpillar, the multigraph $H_1$ is connected.
By definition of~$P$, the multigraph $H_1$ is also Eulerian where all vertices have either degree $2$ or~$4$.

Now suppose we have already constructed $H_i$ and there exists a vertex ${h \in V(H_i)}$ with degree $4$ in $H_i$.
Since $H_i$ is obtained from $H_1$ via repeated splitting operations, we know that $h$ is incident with two edges $d, e$ in $H_i$ that correspond to edges~${d_P, e_P}$, respectively, of ${E(P) \setminus E(T)}$.
Furthermore, $h$ is incident with two edges $f, g$ that correspond to arcs $\overline{A_f}$ and $\overline{A_g}$, respectively, of $\overline{T}$ for subgraphs $A_f$ and $A_g$ of $T$ such that neither $V(A_f) \cup N_T(A_f)$ nor $V(A_g) \cup N_T(A_g)$ contain an endvertex of an edge of ${E(P) \setminus E(T)}$.
Let $T_j$ be the graph-theoretical component of $T$ in which each of $d_P$ and $e_P$ has an endvertex, say $w_d$ and $w_e$, respectively.
Here we consider two cases:

\begin{case}
The distance in $T_j$ between $w_d$ and $w_e$ is even.
\end{case}

In this case we define $H_{i+1}$ as a Eulerian $h$-split of $H_i$ such that one of the following two options holds for the edge $d_{i+1}$ in $H_{i+1}$ corresponding to $d$.
The first option is that $d_{i+1}$ is adjacent to the edge in $H_{i+1}$ corresponding to $e$.
The second options is that $d_{i+1}$ is adjacent to the edge in $H_{i+1}$ corresponding to either $f$ or $g$ with the property that the path in $T_j$ connecting~$w_d$ and $A_f$ (resp.~ $A_g$) does not contain $w_e$.
This is possible since two of the three possible non-isomorphic $v$-splits of $H_i$ are Eulerian by Lemma~\ref{euler_split}.

\begin{case}
The distance in $T_j$ between $w_d$ and $w_e$ is odd.
\end{case}

Here we set $H_{i+1}$ as a Eulerian $h$-split of $H_i$ such that the edge in $H_{i+1}$ corresponding to $d$ is not adjacent to the one corresponding to $e$.
As in the first case, this is possible because two of the three possible non-isomorphic $h$-splits of $H_i$ are Eulerian by Lemma~\ref{euler_split}.
This completes the definition of the sequence of auxiliary multigraphs.
\\

Now we use the last auxiliary multigraph $H_k$ of the sequence to define the arcs $\overline{A^{n+1}}$ and $\overline{B^{n+1}}$.
Note that $P$ is a $w$--$v'$ path in $G[I_{b^n_r, v}]$ where $v'$ and $w$ lie in the same graph-theoretical components $T_{v'}$ and $T_w$ of $T$ as $v$ and $b^n_r$, respectively.
Since we may assume that $E(P) \setminus E(T) \neq \emptyset$ holds, let $e \in E(P) \setminus E(T)$ denote the edge which contains one endvertex $w_e$ in $T_w$.
Then either the distance between $w_e$ and $a^n_r$ or between $w_e$ and $b^n_r$ is even, say the latter one holds.
Now we first extend $B^n$ via a $(b^n_r, w_e]$ square string in $T^2$ and $A^n$ by a $(a^n_r, w_e^+]$ square string in~$T^2$ where $V_{w_e^+}$ is the successor of $V_{w_e}$ with respect to $\leq_T$ and $w_e^+$ is the jumping vertex of~$V_{w_e^+}$.
Then we extend $A^n$ further using a ray to contain all vertices of partition classes $V_x$ with $V_x >_T V_{w_e^+}$ for $x \in T_w$.
This is possible due to the properties~(i) and (ii) of the partition $\mathcal{P}_T$ mentioned in Lemma~\ref{order}.

Next let $P_1$ and $P_2$ be the two edge-disjoint $T_{v'}$--$T_w$ paths in $H_k$.
Since every edge of $E(P) \setminus E(T)$ corresponds to an edge of $H_k$, we get that $e$ corresponds either to $P_1$ or $P_2$, say to the former one.
Therefore, we will use $P_1$ to obtain arcs to extend~$B^n$ and $P_2$ for arcs extending $A^n$.
Now we make use of the definition of $H_k$ via splittings.
For any vertex $T_j$ of $H_1$ of degree $4$ we have performed a $T_j$-split.
We did this in such a way that the partition of the edges incident with $T_j$ into pairs of edges incident with a replacement vertex of $T_j$ corresponds to a cover of $V(T_j)$ via two, possibly infinite, paths as in Lemma~\ref{decomp}.
So for every vertex of $H_1$ of degree $4$ we take such a cover.
For every graph-theoretical component $T_m$ of $T$ such that there exist two consecutive edges $T_iT_j$ and $T_jT_{\ell}$ of $P_1$ or $P_2$ that do not correspond to edges of $E(P) \setminus E(T)$ and $V_{t_i} <_T V_{t_m} <_T V_{t_j}$ or $V_{t_j} <_T V_{t_m} <_T V_{t_{\ell}}$ holds for every choice of $t_i \in T_i$, $t_j \in T_j$, $t_{\ell} \in T_{\ell}$ and $t_m \in T_m$, we take a spanning double ray of $T^2_m$.
We can find such spanning double rays by using again the properties~(i) and (ii) of the partition $\mathcal{P}_T$ mentioned in Lemma~\ref{order}.
Since $H_k = P_1 \cup P_2$ is a cycle, we can use these covers and double rays to extend $\overline{A^n}$ and $\overline{B^n}$ to be disjoint arcs~$\alpha^n$ and $\beta^n$ with endvertices on $T_{v'}$.
With the same construction that we have used for extending $A^n$ and $B^n$ on $T_w$, we can extend $\alpha^n$ and $\beta^n$ to have endvertices~$v'_j$ and $v_j$ which are the jumping vertices of $V_{v'}$ and $V_{v}$, respectively.
Additionally, we incorporate that these extensions contain all vertices of partition classes $V_y$ for $y \in T_{v'}$ and $V_y \leq V_v$.
Then we take these arcs as $\overline{A^{n+1}}$ and $\overline{B^{n+1}}$ where $A^{n+1}$ and $B^{n+1}$ are the corresponding subgraphs of $G^2$ whose closures give the arcs.
By setting $a^{n+1}_r$ and $b^{n+1}_r$ to be $v'_j$ and $v_j$, depending on which of the two arcs $\overline{A^{n+1}}$ or $\overline{B^{n+1}}$ ends in these vertices, we have guaranteed all properties from $(1)$ to $(5)$ for the construction.

Now the properties $(3) - (5)$ yield not only that $\overline{A}$ and $\overline{B}$ are disjoint arcs for ${A = \bigcup_{i \in \mathbb{N}} A^i}$ and ${B = \bigcup_{i \in \mathbb{N}} B^i}$, but also that $V(G) = V(A \cup B)$.
If there exists neither a maximal nor minimal partition class with respect to $\leq_T$, the union~$\overline{A \cup B}$ forms a Hamilton circle of $G^2$ by Lemma~\ref{circ}.
Should there exist a maximal partition class, say $V_{a^n_r}$ for some $n \in \mathbb{N}$ with jumping vertex $a^n_r$, the vertex $a^n_r$ will also be an endvertex of~$\overline{A}$.
In this case we connect the endvertices $a^n_r$ and $b^n_r$ of $\overline{A}$ and~$\overline{B}$ via an edge.
Such an edge exists since $V_{a^n_r}$ and $V_{b^n_r}$ are consecutive with respect to $\leq_T$ by property $(2)$ and $a^n_r$ as well as $b^n_r$ are jumping vertices by property $(1)$.
Analogously, we add an edge if there exists a minimal partition class.
Therefore, we can always obtain the desired Hamilton circle of $G^2$.
\end{proof}

\section{Graphs without $K^4$ or $K_{2,3}$ as minor}

We begin this section with a small observation which allows us to strengthen Theorem~\ref{HC_K_4-K_2,3} a bit by forbidding subgraphs isomorphic to a $K^4$ instead of minors.

\begin{lemma}\label{K^4_minor_subgr}
For graphs without $K_{2, 3}$ as a minor it is equivalent to contain a $K^4$ as a minor or as a subgraph.
\end{lemma}

\begin{proof}
One implication is clear.
So suppose for a contradiction we have a graph without a $K_{2, 3}$ as a minor that does not contain $K^4$ as a subgraph but as a subdivision.
Note that containing a $K^4$ as a subdivision is equivalent to containing a $K^4$ as a minor since $K^4$ is cubic.
Consider a subdivided $K^4$ where at least one edge $e$ of the $K^4$ corresponds to a path $P_e$ in the subdivision whose length is at least two.
Let $v$ be an interior vertex of $P_e$ and $a, b$ be the endvertices of $P_e$.
Let the other two branch vertices of the subdivision of $K^4$ be called $c$ and $d$.
Now we take $\lbrace a, b, c, d, v \rbrace$ as branch vertex set of a subdivision of $K_{2,3}$.
The vertices $a$ and $b$ can be joined to $c$ and $d$ by internally disjoint paths using the ones of the subdivision of $K^4$ except the path $P_e$.
Furthermore, the vertex $v$ can be joined to $a$ and $b$ using the paths $vP_ea$ and $vP_eb$.
So we can find a subdivision of $K_{2,3}$ in the whole graph, which contradicts our assumption.
\end{proof}

Before we start with the proof of Theorem~\ref{HC_K_4-K_2,3} we need to prepare two structural lemmas.
The first one will be very convenient for controlling end degrees because it bounds the size of certain separators.

\begin{lemma}\label{struct_1}
Let $G$ be a $2$-connected graph without $K_{2,3}$ as a minor and let $K_0$ be a connected subgraph of $G$.
Then $|N(K_1)| = 2$ holds for every component $K_1$ of $G-(K_0 \cup N(K_0))$.
\end{lemma}

\begin{proof}
Let $K_0$, $G$ and $K_1$ be defined as in the statement of the lemma.
Since $G$ is $2$-connected, we know that $|N(K_1)| \geq 2$ holds.
Now suppose for a contradiction that $N(K_1) \subseteq N(K_0)$ contains three vertices, say $u, v$ and $w$.
Pick neighbours $u_i$,~$v_i$ and $w_i$ of $u, v$ and $w$, respectively, in $K_i$ for $i \in \lbrace 0, 1 \rbrace$.
Furthermore, take a finite tree $T_i$ in $K_i$ whose leaves are precisely $u_i$, $v_i$ and $w_i$ for $i \in \lbrace 0,1 \rbrace$.
This is possible because $K_0$ and $K_1$ are connected.
Now we have a contradiction since the graph $H$ with $V(H) = \lbrace u, v, w \rbrace \cup V(T_0) \cup V(T_1)$ and $E(H) = \bigcup^1_{i=0} (\lbrace uu_i, vv_i, ww_i \rbrace \cup E(T_i))$ forms a subdivision of $K_{2,3}$.
\end{proof}

Let $G$ be a connected graph and $H$ be a connected subgraph of $G$.
We define the operation of \textit{contracting} $H$ \textit{in} $G$ as taking the minor of $G$ which is attained by contracting in $G$ all edges of $H$.
Now let $K$ be any subgraph of $G$.
We denote by $G_K$ the following minor of $G$:
First contract in $G$ each subgraph that corresponds to a component of~$G-K$.
Then delete all multiple edges.

Obviously $G_K$ is connected if $G$ was connected.
We can push this observation a bit further towards $2$-connectedness with the following lemma.

\begin{lemma}\label{fin_2_con}
Let $K$ be a connected subgraph with at least three vertices of a \linebreak $2$-connected graph $G$.
Then $G_K$ is $2$-connected.
\end{lemma}

\begin{proof}
Suppose for a contradiction that $G_K$ is not $2$-connected for some $G$ and $K$ as in the statement of the lemma.
Since $K$ has at least three vertices, we obtain that $G_K$ has at least three vertices too.
So there exists a cut vertex $v$ in $G_K$.
If $v$ is also a vertex of $G$ and, therefore, does not correspond to a contracted component of $G-K$, then $v$ would also be a cut vertex of $G$.
This contradicts the assumption that $G$ is $2$-connected.

Otherwise $v$ corresponds to a contracted component of $G-K$.
Note that two vertices of $G_K$ both of which correspond to contracted components of $G-K$ are never adjacent by definition of $G_K$.
However, $v$ being a cut vertex in $G_K$ must have at least one neighbour in each component of $G_K-v$.
So in particular we get that $v$ separates two vertices, say $x$ and $y$, of $G_K$ that do not correspond to contracted components of $G-K$.
This yields a contradiction because $K$ is connected and, therefore, contains an $x$--$y$ path.
This path still exists in $G_K$ and contradicts the statement that $v$ separates $x$ and $y$ in~$G_K$.
\end{proof}

We shall need another lemma for the proof Theorem~\ref{HC_K_4-K_2,3}.
In that proof we shall construct an embedding of an infinite graph into a fixed closed disk $D$ by first embedding a finite subgraph into $D$.
Then we extend this embedding stepwise to bigger finite subgraphs so that eventually we define an embedding of the whole graph into $D$.
The following lemma will allow us to redraw newly embedded edges as straight lines in each step while keeping the embedding of every edge that was already embedded as a straight line.
Additionally, we will be able to keep the embedding of those edges that are mapped into the boundary of the disk.

\begin{lemma}\label{straight_embed}
Let $G$ be a finite $2$-connected outerplanar graph and $C$ be its Hamilton cycle.
Furthermore, let $\sigma: G \longrightarrow D$ be an embedding of $G$ into a fixed closed disk $D$ such that $C$ is mapped onto the boundary $\partial D$ of $D$.
Then there is an embedding $\sigma^*: G \longrightarrow D$ such that
\textnormal{
\begin{enumerate}[\normalfont(i)]
\item \textit{$\sigma^*(e)$ is a straight line for every $e \in E(G) \setminus E(C)$.}
\item \textit{$\sigma^*(e) = \sigma(e)$ if $e \in E(C)$ or $\sigma(e)$ is a straight line.}
\end{enumerate}
}\end{lemma}

\begin{proof}
We prove the statement by induction on $\ell := |E(G) \setminus E(C)|$.
For $\ell = 0$ we can choose the given embedding $\sigma$ as our desired embedding $\sigma^*$.
Now let $\ell \geq 1$ and suppose $\sigma$ does not already fulfill all properties of $\sigma^*$.
Then there exists an edge $e \in E(G) \setminus E(C)$ such that $\sigma(e)$ is not a straight line.
Hence, $G-e$ is still a $2$-connected outerplanar graph that contains $C$ as its Hamilton cycle.
Also $\sigma \upharpoonright _{G-e}$ is an embedding of $G-e$ into $D$ such that $C$ is mapped onto $\partial D$.
So by the induction hypothesis we get an embedding $\tilde{\sigma}^*$ satisfying~(i)~and~(ii) with respect to $\sigma \upharpoonright _{G-e}$.
Now let $e = uv$ and suppose for a contradiction that we cannot additionally embed $e$ as a straight line between $u$ and $v$.
Then there exists an edge $xy \in E(G-e) \setminus E(C)$ such that $\tilde{\sigma}^*(xy)$ is crossed by the straight line between $u$ and $v$.
Because $\tilde{\sigma}^*(xy)$ is a straight line between $x$ and $y$ by property~(ii), we know that the vertices $u, v, x$ and $y$ are pairwise distinct.
This, however, is a contradiction to $G$ being outerplanar since the cycle $C$ together with the edges $uv$ and $xy$ witness the existence of a $K^4$ minor in $G$ with $u, v, x$ and $y$ as branch sets.
So we can extend $\tilde{\sigma}^*$ by embedding $e = uv$ as a straight line between $u$ and $v$, which yields our desired embedding of $G$ into $D$.
\end{proof}

With the lemmas above we are now prepared to prove Theorem~\ref{HC_K_4-K_2,3}.
We recall the formulation of the theorem.

\setcounter{section}{1}
\setcounter{theorem}{7}
\begin{theorem}
Let $G$ be a locally finite connected graph.
Then the following statements are equivalent:
\begin{enumerate}[\normalfont(i)]
\item $G$ is $2$-connected and contains neither $K^4$ nor $K_{2,3}$ as a minor.
\item $|G|$ has a Hamilton circle $C$ and there exists an embedding of $|G|$ into a closed disk such that $C$ is mapped onto the boundary of the disk.
\end{enumerate}
Furthermore, if statements (i) and (ii) hold, then $|G|$ has a unique Hamilton circle.
\end{theorem}
\setcounter{theorem}{4}
\setcounter{section}{4}

\begin{proof}
First we show that $(ii)$ implies $(i)$.
Since $G$ is Hamiltonian, we know by Corollary~\ref{2-con} that $G$ is $2$-connected.
Suppose for a contradiction that $G$ contains $K^4$ or $K_{2,3}$ as a minor.
Then $G$ has a finite subgraph $H$ which already has $K^4$ or $K_{2,3}$ as a minor.
Now take any finite connected subgraph $K_0$ of $G$ which contains~$H$ and set~${K = G[V(K_0) \cup N(K_0)]}$.
Next let us take an embedding of $|G|$ as in statement~$(ii)$ of this theorem.
It is easy to see using Lemma~\ref{struct_1} that our fixed embedding of $|G|$ induces an embedding of $G_K$ into a closed disk such that all vertices of~$G_K$ lie on the boundary of the disk.
This implies that $G_K$ is outerplanar.
So~$G_K$ can neither contain $K^4$ nor $K_{2,3}$ as a minor by Theorem~\ref{outerplanar_count_char}, which contradicts that $H$ is a subgraph of $G_K$.

Now let us assume $(i)$ to prove the remaining implication.
We set $K_0$ as an arbitrary connected subgraph of $G$ with at least three vertices.
Next we define~${K_{i+1} = G[V(K_i) \cup N(K_i)]}$ for every ${i \geq 0}$.
Inside $G$ we define the vertex sets ${L_{i} = \lbrace v \in V(K_i) \; ; \; N(v) \subseteq V(K_i) \rbrace}$ for every $i \geq 1$.
Let then ${\tilde{K}_{i+1} = G_{K_{i+1}} - L_i}$ for every ${i \geq 1}$.
By Lemma~\ref{fin_2_con} we know that $G_{K_i}$ is $2$-connected for each ${i \geq 0}$.
Furthermore, $G_{K_i}$ contains neither $K^4$ nor $K_{2,3}$ as a minor for every ${i \geq 0}$ since it would also be a minor of $G$ contradicting our assumption.
So each $G_{K_i}$ is outerplanar by Theorem~\ref{outerplanar_count_char}.
Using statement~(ii) of Proposition~\ref{summary} we obtain that each $G_{K_i}$ has a unique Hamilton cycle $C_i$ and that there is an embedding $\sigma_i$ of $G_{K_i}$ into a fixed closed disk $D$ such that $C_i$ is mapped onto the boundary~$\partial D$ of $D$.
Set ${E_i = E(C_i) \cap  E(K_i)}$ for every ${i \geq 1}$.

Next we define an embedding of $G$ into $D$ and extend it to the desired embedding of $|G|$.
We start by taking $\sigma_1$.
Note again that $G_{K_1}$ is a finite $2$-connected outerplanar graph by Lemma~\ref{fin_2_con}.
Furthermore, $\sigma_1(C_1) = \partial D$ .
So we can use Lemma~\ref{straight_embed} to obtain an embedding $\sigma^*_1: G_{K_1} \longrightarrow D$ as in the statement of that lemma.
Because of Lemma~\ref{struct_1} we can extend $\sigma^*_1 \upharpoonright _{K_1}$ using $\sigma_2 \upharpoonright _{\tilde{K}_2}$, maybe after rescaling the latter embedding, to obtain an embedding $\varphi_2: G_{K_2} \longrightarrow D$ such that $\varphi_2(C_2) = \partial D$.
We apply again Lemma~\ref{straight_embed} with $\varphi_2$, which yields an embedding $\sigma^*_2: G_{K_2} \longrightarrow D$ as in the statement of that lemma.
Note that this construction ensures $\sigma^*_2 \upharpoonright _{K_1} = \sigma^*_1 \upharpoonright _{K_1}$.
Proceeding in the same way, we get an embedding $\sigma^*:G \longrightarrow D$ by setting $\sigma^* := \bigcup_{i \in \mathbb{N}} \sigma^*_i \upharpoonright _{K_i}$.
The use of Lemma~\ref{straight_embed} in the construction of $\sigma^*$ ensures that all edges are embedded as straight lines unless they are contained in any $E_i$.
However, all edges in the sets $E_i$, and therefore also all vertices of $G$, are embedded into $\partial D$.
Furthermore, we may assure that $\sigma^*$ has the following property:

\begin{table}[h]
\centering
\begin{tabular}{cl}
\begin{minipage}{0.9\textwidth}
\textit{Let $(M_i)_{i \geq 1}$ be any infinite sequence of components $M_i$ of ${G - K_i}$ where ${M_{i+1} \subseteq M_i}$. Also, let $\lbrace u_i, w_i \rbrace$ be the neighbourhood of $M_i$ in $G$.
Then the sequences $(\sigma^*(u_i))_{i \geq 1}$ and $(\sigma^*(w_i))_{i \geq 1}$ converge to a common point on $\partial D$.}
\end{minipage} & $(\ast)$
\end{tabular}
\end{table}

It remains to extend this embedding $\sigma^*$ to an embedding $\overline{\sigma}^*$ of all of $|G|$ into $D$.
First we shall extend the domain of $\sigma^*$ to all of $|G|$.
For this we need to prove the following claim.

\begin{claim}
For every end $\omega$ of $G$ there exists an infinite sequence $(M_i)_{i \geq 1}$ of components $M_i$ of ${G - K_i}$ with ${M_{i+1} \subseteq M_i}$ such that $\bigcap_{i \geq 1}\overline{M_i} = \lbrace \omega \rbrace$.
\end{claim}

Since $K_i$ is finite, there exists a unique component of ${G - K_i}$ in which all $\omega$-rays have a tail.
Set this component as $M_i$.
It follows from the definition that $\omega$ lies in~$\overline{M_i}$.
Furthermore, we get that $\bigcap_{i \geq 1}\overline{M_i}$ does neither contain any vertex nor an inner point of any edge.
So suppose for a contradiction that $\bigcap_{i \geq 1}\overline{M_i}$ contains another end~$\omega' \neq \omega$.
We know there exists a finite set $S$ of vertices such that all tails of $\omega$-rays lie in a different component of $G-S$ than all tails of $\omega'$-rays.
By definition of the graphs~$K_i$ we can find an index $j$ such that $S \subseteq V(K_j)$.
So $\omega$ lies in $\overline{M_j}$ and $\omega'$ in~$\overline{M'_j}$ where~$M'_j$ is the component of $G-K_j$ in which all tails of $\omega'$-rays lie.
Since $G$ is locally finite, the cut $E(M_j, K_j)$ is finite.
Using Lemma~\ref{jumping-arc} we obtain that $\overline{M_j} \cap \overline{M'_j} = \emptyset$.
Therefore, $\omega' \notin \overline{M_j} \supseteq \bigcap_{i \geq 1}\overline{M_i}$.
This contradiction completes the proof of the claim.
\newline

Now let us define the map $\overline{\sigma}^*$.
For every vertex or inner point of an edge $x$, we set~$\overline{\sigma}^*(x) = \sigma^*(x)$.
For an end $\omega$ let $(M_i)_{i \geq 1}$ be the sequence of components $M_i$ of~${G - K_i}$ given by Claim~1 and $\lbrace u_i, w_i \rbrace$ be the neighbourhood of $M_i$ in $G$.
Using property~$(\ast)$ we know that $(\sigma^*(u_i))_{i \geq 1}$ and $(\sigma^*(w_i))_{i \geq 1}$ converge to a common point~$p_{\omega}$ on $\partial D$.
We use this to set $\overline{\sigma}^*(\omega) = p_{\omega}$.
This completes the definition of~$\overline{\sigma}^*$.

Next we prove the continuity of $\overline{\sigma}^*$.
For every vertex or inner point of an edge $x$, it is easy to see that an open set around $\overline{\sigma}^*(x)$ in $D$ contains $\overline{\sigma}^*(U)$ for some open set~$U$ around $x$ in $|G|$.
This holds because $G$ is locally finite and so it follows from the definition of $\overline{\sigma}^*$ using the embeddings $\sigma^*_i$.
Let us check continuity for ends.
Consider an open set $O$ around $\overline{\sigma}^*(\omega)$ in $D$, where $\omega$ is an end of $G$.
Let $B_{\varepsilon}(\overline{\sigma}^*(\omega))$ denote the restriction to $D$ of an open ball around $\overline{\sigma}^*(\omega)$ with radius $\varepsilon > 0$.
Then $B_{\varepsilon}(\overline{\sigma}^*(\omega))$ is an open set and, for sufficiently small $\varepsilon$, contained in~$O$.
We fix such an $\varepsilon$ for the rest of this proof.
Let $(M_i)_{i \geq 1}$ be a sequence as in Claim~1 for $\omega$ and $\lbrace u_i, w_i \rbrace$ be the neighbourhood of $M_i$ in $G$.
By property $(\ast)$ and the definition of $\overline{\sigma}^*$, we get that $(\sigma^*(u_i))_{i \geq 1}$ and $(\sigma^*(w_i))_{i \geq 1}$ converge to $\overline{\sigma}^*(\omega)$ on $\partial D$.
So there exists a $j \in \mathbb{N}$ such that $B_{\varepsilon}(\overline{\sigma}^*(\omega))$ contains $\sigma^*(u_i)$ and $\sigma^*(w_i)$ for every $i \geq j$.
By the definitions of $\overline{\sigma}^*$ and $\sigma^*$ using the embeddings $\sigma^*_i$, it follows that $\overline{\sigma}^*(\overline{M_j}) \subsetneqq B_{\varepsilon}(\overline{\sigma}^*(\omega)) \subseteq O$.
At this point we use the property of $\sigma^*$ that every edge of $G$ is embedded as a straight line unless it is embedded into $\partial D$.
Hence, if $vw \in E(G)$ and $\overline{\sigma}^*(v), \overline{\sigma}^*(w) \in B_{\varepsilon}(\overline{\sigma}^*(\omega))$, then $\overline{\sigma}^*(vw)$ is also contained in $B_{\varepsilon}(\overline{\sigma}^*(\omega))$ by the convexity of the ball.
Since $\overline{M_j}$ together with the inner points of the edges of $E(M_j, K_j)$ is a basic open set in~$|G|$ containing $\omega$ whose image under~$\overline{\sigma}^*$ is contained in $O$, continuity holds for ends too.

The next step is to check that $\overline{\sigma}^*$ is injective.
If $x$ and $y$ are each either a vertex or an inner point of an edge, then they already lie in some $K_j$.
By the definition of $\overline{\sigma}^*$ we get that $\overline{\sigma}^*(x) = \overline{\sigma}^*(y)$ if and only if there exists a $j \in \mathbb{N}$ such that $x$ and $y$ are mapped to the same point by the embedding of $K_j$ defined by $\bigcup^j_{i=1} \sigma^*_{i} \upharpoonright _{K_{i}}$.
So $x$ and $y$ need to be equal.

For an and $\omega$ of $G$, let $(M_i)_{i \geq 1}$ be a sequence of components of $G-K_i$ such that $\bigcap_{i \geq 1}\overline{M_i} = \lbrace \omega \rbrace$, which exists by Claim~1.
Let $\lbrace u_i, w_i \rbrace$ be the neighbourhood of $M_i$ in $G$.
Since $G$ is locally finite, there exists an integer $j$ such that $y$ lies in $K_j$ if it is a vertex or an inner point of an edge, or $y$ lies in $\overline{M'_j}$ for some component $M'_j \neq M_j$ of $G-K_j$ if $y$ is an end of $G$ that is different from $\omega$.
By the definition of $\overline{\sigma}^*$ and property $(\ast)$ we get that the arc on $\partial D$ between $\sigma^*(u_j)$ and $\sigma^*(w_j)$ into which the vertices of $M_j$ are mapped contains also $\overline{\sigma}^*(\omega)$ but not $y$.
Hence, $\overline{\sigma}^*(\omega) \neq \overline{\sigma}^*(y)$ if $\omega \neq y$.
This shows the injectivity of the map $\overline{\sigma}^*$.

To see that the inverse function of ${\overline{\sigma}^*}$ is continuous, note that $|G|$ is compact by Proposition~\ref{compact} and~$D$ is Hausdorff.
So Lemma~\ref{invers_cont} immediately implies that the inverse function of ${\overline{\sigma}^*}$ is continuous.
This completes the proof that $\overline{\sigma}$ is an embedding.

It remains to show the existence of a unique Hamilton circle of $G$ that is mapped onto $\partial D$ by $\overline{\sigma}$.
For this we first prove that $\partial D \subseteq \mathrm{Im}(\overline{\sigma})$.
This then implies that the inverse function of ${\overline{\sigma}^*}$ restricted to $\partial D$ is a homeomorphism defining a Hamilton circle of $G$ since it contains all vertices of $G$.
We begin by proving the following claim.


\begin{claim}
For every infinite sequence $(M_i)_{i \geq 1}$ of components $M_i$ of ${G - K_i}$ with ${M_{i+1} \subseteq M_i}$ there exists an end $\omega$ of $G$ such that $\bigcap_{i \geq 1}\overline{M_i} = \lbrace \omega \rbrace$.
\end{claim}

Let $(M_i)_{i \geq 1}$ be any sequence as in the statement of the claim.
Since for every vertex $v$ there exists a $j \in \mathbb{N}$ such that $v \in K_j$, we get that $\bigcap_{i \geq 1}\overline{M_i}$ is either empty or contains ends of $G$.
Using that each $M_i$ is connected and that $M_{i+1} \subseteq M_i$, we can find a ray $R$ such that every $M_i$ contains a tail of $R$.
Therefore, $\bigcap_{i \geq 1}\overline{M_i}$ contains the end in which $R$ lies.
The argument that $\bigcap_{i \geq 1}\overline{M_i}$ contains at most one end is the same as in the proof of Claim~1.
This completes the proof of Claim~2.
\newline

Suppose a point $p \in \partial D$ does not already lie in $\mathrm{Im}(\sigma^*)$.
Then it does not lie in $\mathrm{Im}(\sigma^*_i \upharpoonright _{K_i})$ for any $i \geq 1$.
So there exists an infinite sequence $(M_i)_{i \geq 1}$ of components $M_i$ of ${G - K_i}$ with ${M_{i+1} \subseteq M_i}$ such that $p$ lies in the arc~$A_i$ of $\partial D$ between $\sigma^*(u_i)$ and $\sigma^*(w_i)$ into which the vertices of $M_i$ are mapped, where $\lbrace u_i, w_i \rbrace$ denotes the neighbourhood of $M_i$ in $G$.
Using Claim~2 we obtain that there exists an end $\omega$ of $G$ such that $\bigcap_{i \geq 1}\overline{M_i} = \lbrace \omega \rbrace$.
By property $(\ast)$ of the map~$\sigma^*$ the sequences $(\sigma^*(u_i))_{i \geq 1}$ and $(\sigma^*(w_i))_{i \geq 1}$ converge to a common point on~$\partial D$.
This point must be $p$ since the arcs $A_i$ are nested.
Now the definition of $\overline{\sigma}^*$ tells us that~$\overline{\sigma}^*(\omega) = p$.
Hence $\partial D \subseteq \mathrm{Im}(\overline{\sigma}^*)$ and $G$ is Hamiltonian.

We finish the proof by showing the uniqueness of the Hamilton circle of $G$.
Suppose for a contradiction that $G$ has two subgraphs $C_1$ and $C_2$ yielding different Hamilton circles $\overline{C_1}$ and $\overline{C_2}$.
Then there must be an edge $e \in E(C_1) \setminus E(C_2)$.
Let $j \in \mathbb{N}$ be chosen such that $e \in E(K_j)$.
By Lemma~\ref{struct_1} we obtain that $G_{K_j}[E(C_1) \cap E(G_{K_j})]$ and $G_{K_j}[E(C_2) \cap E(G_{K_j})]$ are two Hamilton cycles of $G_{K_j}$ differing in the edge $e$.
Note that $G_{K_j}$ is a finite $2$-connected outerplanar graph.
The argument for this is the same as for $G_K$ in the proof that $(ii)$ implies $(i)$.
This yields a contradiction since $G_{K_j}$ has a unique Hamilton cycle by statement~(ii) of Proposition~\ref{summary}.
\end{proof}

Next we deduce Corollary~\ref{Cor_contr}.
Let us recall its statement first.

\setcounter{section}{1}
\setcounter{theorem}{8}
\begin{corollary}
The edges contained in the Hamilton circle of a locally finite \linebreak $2$-connected graph not containing $K^4$ or $K_{2,3}$ as a minor are precisely the \linebreak $2$-contractible edges of the graph unless the graph is isomorphic to a $K^3$.
\end{corollary}
\setcounter{theorem}{4}
\setcounter{section}{4}

\begin{proof}
Let $G$ be a locally finite $2$-connected graph not isomorphic to a $K^3$ and not containing $K^4$ or $K_{2,3}$ as a minor.
Further, let $C$ be the subgraph of $G$ such that $\overline{C}$ is the Hamilton circle of $G$.
First we show that each edge $e \in E(C)$ is a $2$-contractible edge.
Note for this that the closure of the subgraph of $G/e$ formed by the edge set $E(C) \setminus \lbrace e \rbrace$ is a Hamilton circle in $|G/e|$.
Hence, $G/e$ is $2$-connected by Corollary~\ref{2-con}.

It remains to verify that no edge of $E(G) \setminus E(C)$ is $2$-contractible.
For this we consider any edge $e = uv \in E(G) \setminus E(C)$.
Let $K$ be a finite connected induced subgraph of $G$ containing at least four vertices as well as $N(u) \cup N(v)$, which is a finite set since $G$ is locally finite.
Then we know by Lemma~\ref{fin_2_con} and by using the locally finiteness of $G$ again that $G_K$ is a finite $2$-connected graph not containing $K^4$ or $K_{2,3}$ as a minor.
So by Theorem~\ref{outerplanar_count_char} and Proposition~\ref{summary} we get that $G_K$ has a unique Hamilton cycle consisting precisely of its $2$-contractible edges.
However, as we have seen in the proof of Theorem~\ref{HC_K_4-K_2,3}, $G_K[E(C) \cap E(G_K)]$ is the unique Hamilton cycle of $G_K$ and does not contain $e$.
Since $G_K$ is outerplanar, we get that the vertex of $G_K/e$ corresponding to the edge $e$ is a cut vertex in $G_K/e$.
By our choice of $K$ containing $N(u) \cup N(v)$, we get that the vertex in $G/e$ corresponding to the edge $e$ is a cut vertex of $G/e$ too.
So $e$ is not $2$-contractible.
\end{proof}

The question arises whether one could prove the more complicated part of Theorem~\ref{HC_K_4-K_2,3}, the implication $(i) \Longrightarrow (ii)$, by mimicking a proof for finite graphs.
To see the positive answer for this question, let us summarize the proof for finite graphs except the part about the uniqueness.

By Theorem~\ref{outerplanar_count_char} every finite graph without $K^4$ or $K_{2,3}$ as a minor can be embedded into the plane such that all vertices lie on a common face boundary.
Since every face of an embedded $2$-connected graph is bounded by a cycle, we obtain the desired Hamilton cycle.

So for our purpose we would first need to prove a version of Theorem~\ref{outerplanar_count_char} for~$|G|$ where $G$ is a locally finite connected graph.
This can similarly be done in the way we have defined the embedding for the Hamilton circle in Theorem~\ref{HC_K_4-K_2,3} by decomposing the graph into finite parts using Lemma~\ref{struct_1}.
Since none of these parts contains a $K^4$ or a $K_{2,3}$ as a minor, we can fix appropriate embeddings of them and stick them together.
However, in order to obtain an embedding of $|G|$ we have to be careful.
We also need to ensure that the embeddings of finite parts that converge to an end in $|G|$ also converge to a point in the plane where we can map the corresponding end to.

The second ingredient of the proof is the following lemma pointed out by Bruhn and Stein, but which is a corollary of a stronger and more general result of Richter and Thomassen~\cite[Prop.\ 3]{richter}.

\begin{lemma}\label{circ_boundary}\cite[Cor.\ 21]{bruhn_stein}
Let $G$ be a locally finite $2$-connected graph with an embedding ${\varphi : |G| \longrightarrow S^2}$.
Then the face boundaries of $\varphi(|G|)$ are circles of $|G|$.
\end{lemma}

\noindent These observations show that the proof idea for finite graphs is still applicable for locally finite graphs.

Let us compare the proof for the implication $(i) \Longrightarrow (ii)$ of Theorem~\ref{HC_K_4-K_2,3} that we sketched right above, with the one we outlined completely.
The two proofs share a big similarity.
Both need to show first that $|G|$ can be embedded into the plane such that all vertices lie on a common face boundary if $G$ is a connected or $2$-connected, respectively, locally finite graph without $K^4$ or $K_{2,3}$ as a minor.
At this point the proof we outlined completely already incorporates further properties into the embedding without too much additional effort.
Especially, we use the $2$-connectedness of the graph there by finding suitable finite $2$-connected contraction minors.
Then we apply Proposition~\ref{summary} for these.
The embeddings we obtain for the contraction minors allow us to define an embedding of $|G|$ into a fixed closed disk.
Furthermore, this embedding of $|G|$ has the additional property that its restriction onto the boundary of the disk directly witnesses the existence of a Hamilton circle.
The second proof, however, takes a step backward and argues more general.
There the $2$-connectedness of $G$ is used to apply Lemma~\ref{circ_boundary}, which, as noted before, is a corollary of a more general result of Richter and Thomassen~\cite[Prop.\ 3]{richter}.
At this point we forget about the special embedding of $|G|$ into the plane that we had to construct before.
We continue the argument with an arbitrary one given that $G$ is a $2$-connected locally finite graph.
So for the purpose of proving the implication $(i) \Longrightarrow (ii)$ of Theorem~\ref{HC_K_4-K_2,3}, the outlined proof is more straightforward and self-contained.

\section{A cubic infinite graph with a unique Hamilton circle}

This section is dedicated to Theorem~\ref{Q_mohar_Yes}.
We shall construct an infinite graph with a unique Hamilton circle where all vertices in the graph have degree $3$.
Furthermore, all ends of that graph have vertex-degree $3$ as well as edge-degree $3$.
The main ingredient in our construction is the finite graph $T$ depicted in Figure~\ref{Tutte_HCs}.
This graph has three distinguished vertices of degree $1$, which we denote by $u$, $l$ and $r$ as in Figure~\ref{Tutte_HCs}.
For us, the important feature of $T$ is that we know where all \textit{Hamilton paths}, i.e., spanning paths, of $T-u$ and $T-r$ proceed.
Tutte~\cite{tutte} came up with the graph $T$ to construct a counterexample to Tait's conjecture~\cite{tait}, which said that every $3$-connected cubic planar graph is Hamiltonian.
The crucial observation of Tutte in~\cite{tutte} was that $T-u$ does not contain a Hamilton path.
We shall use this observation as well, but we need more facts about $T$, which are covered in the following lemma.
The proof is straightforward, but involves several cases that need to be distinguished.

\begin{lemma}\label{HCs_in_T}
There is no Hamilton path in $T-u$, but there are precisely two in $T-r$ (see Figure~\ref{Tutte_HCs}).
\end{lemma}

\begin{figure}[htbp]
\centering
\includegraphics{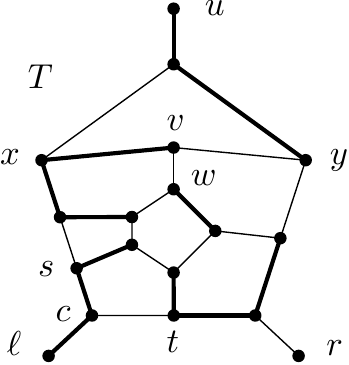}
\hspace{20pt}
\includegraphics{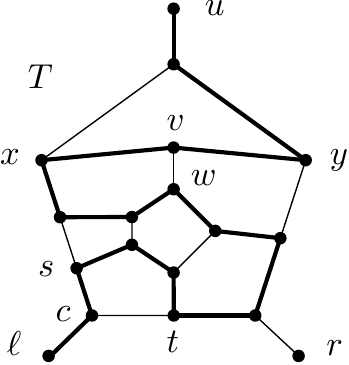}
\hspace{20pt}
\includegraphics{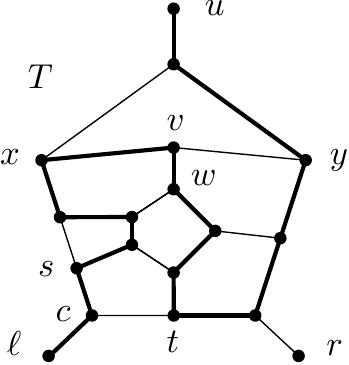}
\caption{The fat edges in the most left picture are in every Hamilton path of $T-r$. The fat edges in the other two pictures mark the two Hamilton paths of $T-r$.}
\label{Tutte_HCs}
\end{figure}

\begin{proof}
As mentioned already by Tutte~\cite{tutte}, the graph $T-u$ does not have a Hamilton path.
It remains to show that $T-r$ has precisely two Hamilton paths.
For this we need to check several cases, but afterwards we can precisely state the Hamilton paths.
For convenience, we label each edge with a number as depicted in Figure~\ref{Tutte_edges} and refer to the edges just by their labels for the rest of the proof.

\begin{figure}[htbp]
\centering
\includegraphics[width=55mm]{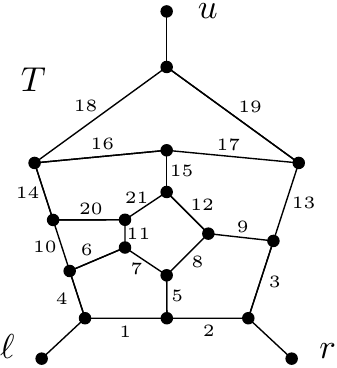}
\caption{Our fixed labelling of the relevant edges of $T$.}
\label{Tutte_edges}
\end{figure}

Obviously, the edges incident with $\ell$ and $u$ would need to be in every Hamilton path of $T-r$ since these vertices have degree $1$.
Furthermore, the edges $2$ and $3$ need to be in every Hamilton path of $T-r$ since the vertex incident with $2$ and $3$ has degree $2$ in $T-r$.

\begin{claim}
The edge $4$ needs to be in every Hamilton path of $T-r$.
\end{claim}

Suppose for a contradiction that there is a Hamilton path $P$ in $T-r$ that does not use $4$.
Then it needs to contain $1$.
Since it also contains $2$, we know $5 \notin E(P)$.
This implies further that $7,8 \in E(P)$.
We can use $4 \notin E(P)$ also to deduce that $6, 10 \in E(P)$ holds.
Now we get $11 \notin E(P)$ since $6, 7 \in E(P)$.
This implies $20, 21 \in E(P)$.
But now $14 \notin E(P)$ holds because $10, 20 \in E(P)$.
From this we get then $16, 18 \in E(P)$.
So $19$ cannot be contained in $P$, which implies $13, 17 \in E(P)$.
Now we arrived at a contradiction since the edges incident with $l$ and $u$ together with the edges of the set $\lbrace 1, 2, 3, 13,17, 16, 18 \rbrace$ form a $\ell$-$u$ path in $T-r$ that is contained in $P$ and needs therefore to be equal to $P$.
Then, however, $P$ would not be a Hamilton path $T-r$.
This completes the proof of Claim~3
\newline

We immediately get from Claim~3 that $5$ needs to be in every Hamilton path of $T-r$ and since $8$ and $9$ can not both be contained in any Hamilton path of $T-r$, because they would close a cycle together with $5, 2$ and $3$, we also know that $12$ needs to be in every Hamilton path of $T-r$.

\begin{claim}
The edges $14$ and $16$ lie in every Hamilton path of $T-r$.
\end{claim}

Suppose for a contradiction that the claim is not true.
Then there is a Hamilton path $P$ of $T-r$ containing $18$.
So $P$ cannot contain $19$, which implies $13, 17 \in E(P)$.
Since $3, 13 \in E(P)$, we obtain $9 \notin E(P)$, from which we follow that $8 \in P$ holds.
Furthermore, $15$ cannot be contained in $P$, because then the edges $15, 17, 13, 3, 2, 5, 8, 12$ would form a cycle in $P$.
Therefore, $16$ is an edge of $P$.
From $5, 8 \in E(P)$ we can deduce that $7 \notin E(P)$ holds.
So $6$ and $11$ are edges of $P$, which that implies $10 \notin E(P)$.
Then $14, 20 \in E(P)$ needs to be true.
Now, however, we have a contradiction, because $P$ would have a vertex incident with three vertices, namely $14, 16$ and $18$.
This completes the proof of Claim~4
\newline

It follows from Claim~4 that $19$ is contained in every Hamilton path of $T-r$.
We continue with another claim.

\begin{claim}
The edges $6$ and $20$ lie in every Hamilton path of $T-r$.
\end{claim}

Suppose for a contradiction that the claim is not true.
Then there is a Hamilton path $P$ of $T-r$ containing $10$.
This immediately implies that $6 \notin E(P)$, yielding $7, 11 \in E(P)$, and $20 \notin E(P)$, yielding $21 \in E(P)$.
We note that $8$ cannot be an edge of $P$ since $P$ would then contain a cycle spanned by the edge set $\lbrace 8, 7, 11, 21, 12 \rbrace$.
Therefore, $9 \in E(P)$ must hold.
Here we arrive at a contradiction, since $P$ now contains a cycle spanned by the edge set $\lbrace 9, 3, 2, 5, 7, 11, 21, 12 \rbrace$.
This completes the proof of Claim~5
\newline

Using all the observations we have made so far, we can now show that $T-r$ has precisely two Hamilton paths and state them by looking at the edge $11$.
Assume that $11$ is contained in a Hamilton path $P_1$ of $T-r$.
Then $7, 21 \notin E(P_1)$ follows, because $6, 20 \in E(P_1)$ holds by Claim~5.
Since we could deduce from Claim~3 that $5, 12 \in E(P_1)$ holds, we get furthermore $8, 15 \in E(P_1)$.
This now yields $9, 17 \notin E(P_1)$ and, therefore, $13 \in E(P_1)$.
As we can see, the assumption that $11$ is contained in a Hamilton path $P_1$ of $T-r$ is true.
Also, $P_1$ is uniquely determined with respect to this property and consists of the fat edges in the most right picture of Figure~\ref{Tutte_HCs}.

Next assume that there is a Hamilton path $P_2$ of $T-r$ that does not contain the edge $11$.
Then $7$ and $21$ have to be edges of $P_2$.
Using again that $5, 12 \in E(P_2)$ holds, we deduce $8, 15 \notin E(P_2)$.
Then, however, we get $9, 17 \in E(P)$ and have already uniquely determined $P_2$, which corresponds to the fat edges in the middle picture of Figure~\ref{Tutte_HCs}.
\end{proof}

Using Lemma~\ref{HCs_in_T} we shall now prove Theorem~\ref{Q_mohar_Yes} by constructing a prescribed graph.
During the construction we shall often refer to certain distinguished vertices of $T$ that are named as depicted in Figure~\ref{Tutte_HCs}.
Let us recall the statement of the theorem.

\setcounter{section}{1}
\setcounter{theorem}{10}
\begin{theorem}
There exists an infinite connected graph $G$ with a unique Hamilton circle that has degree $3$ at every vertex and vertex- as well as edge-degree $3$ at every end.
\end{theorem}
\setcounter{theorem}{1}
\setcounter{section}{5}

\begin{proof}
We construct a sequence of graphs $(G_n)_{n \in \mathbb{N}}$ inductively and obtain the desired one $G$ as a limit of the sequence.
We start with $G_0 = T^1_0 = T$.

Now suppose we have already constructed $G_n$ for $n \geq 0$.
Furthermore, let ${\lbrace T^i_n \; ; \; 1 \leq i \leq 2^n \rbrace}$ be a specified set of disjoint subgraphs of $G_n$ each of which each is isomorphic to $T$.
We define $G_{n+1}$ as follows.
Take $G_n$ and two copies $T_c$ and $T_v$~of $T$ for each $T^i_n \subseteq G_n$.
Then identify for every $i$ the vertices of $T_c$ that correspond to $u$, $\ell$ and $r$, respectively, with the vertices of the related $T^i_n \subseteq G_n$ corresponding to $\ell$, $s$ and~$t$, respectively.
Also identify for every $i$ the vertices of $T_v$ corresponding to $u$,~$\ell$ and~$r$, respectively, with the ones of the related $T^i_n \subseteq G_n$ corresponding to~$w$,~$x$ and~$y$, respectively.
Finally, delete in each $T^i_n \subseteq G_n$ the vertices corresponding to~$c$ and~$v$, see Figure~\ref{Tutte_insert_constr}.
This completes the definition of~$G_{n+1}$.
It remains to fix the set of $2^{n+1}$ many disjoint copies of $T$ that occur as disjoint subgraphs in $G_{n+1}$.
For this we take the set of all copies $T_c$ and $T_v$ of $T$ that we have inserted in the subgraphs $T^i_n$ of $G_n$.

\begin{figure}[htbp]
\centering
\includegraphics{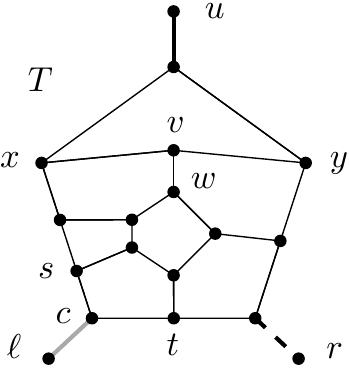}
\hspace{20pt}
\includegraphics{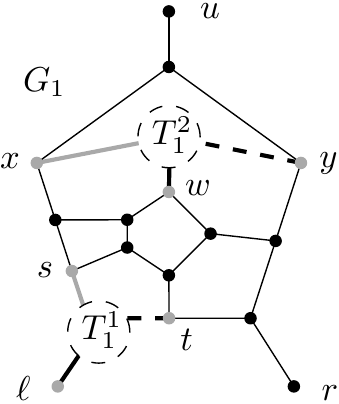}
\caption{A sketch of the construction of $G_1$. The fat black, grey and dashed edges incident with the grey vertices in the right picture correspond to the ones in the left picture.}
\label{Tutte_insert_constr}
\end{figure}

Using the graphs $G_n$ we define a graph $\hat{G}$ as a limit of them.
We set
\[\hat{G} = G[\hat{E}] \; \textnormal{ where } \; \hat{E} = \left \lbrace e \in \bigcup_{n \in \mathbb{N}} E(G_n) \; ; \; \exists N \in \mathbb{N} : e \in \bigcap_{n \geq N} E(G_n) \right \rbrace .\]
Note that an edge $e \in E(G_n)$ is an element of $\hat{E}$ if and only if it was not deleted during the construction of $G_{n+1}$ as an edge incident with one of the vertices that correspond to $c$ or $v$ in $T^i_n$ for some $i$.
Finally, we define $G$ as the graph obtained from $\hat{G}$ by identifying the three vertices that correspond to $u$, $\ell$ and $r$ of $T^1_0$.

Next let us verify that every vertex of $G$ has degree $3$ and that every end of $G$ has vertex- as well as edge-degree $3$ in $G$.
Since every vertex of $T$ except $u$, $\ell$ and $r$ has degree $3$, the construction ensures that every vertex of $G$ has degree $3$ too.
In order to analyse the end degrees, we have to make some observations first.
The edges of $G$ that are adjacent to vertices corresponding to $u$, $\ell$ and $r$ of any $T^i_n$ define a cut $E(A^i_n, B^i_n)$ of $G$.
Note that for any finite cut of a graph all rays in one end of the graph have tails that lie completely on one side of the cut.
Therefore, the construction of $G$ ensures that for every end $\omega$ of $G$ there exists a function $f: \mathbb{N} \longrightarrow \mathbb{N}$ with $f(n) \in \lbrace 1, \ldots, 2^n \rbrace$ such that all rays in $\omega$ have tails in $B^{f(n)}_n$ for each $n \in \mathbb{N}$ and $B^{f(n)}_n \supseteq B^{f(n+1)}_{n+1}$ with $\bigcap_{n \in \mathbb{N}} B^{f(n)}_n = \emptyset$.
Using that $|E(A^i_n, B^i_n)| = 3$ for every $n$ and $i$, this implies that every end of $G$ has edge-degree at most $3$.
Since there are three disjoint paths from $\lbrace u, \ell, r \rbrace$ to $\lbrace s, \ell, t \rbrace$ as well as to $\lbrace x, w, y \rbrace$ in~$T$, we can also easily construct three disjoint rays along the cuts $E(A^i_n, B^i_n)$ that belong to an arbitrary chosen end of $G$.
So every end of $G$ has vertex-degree $3$.
In total this yields that every end of $G$ has vertex- as well as edge-degree $3$ in $G$.

It remains to prove that $G$ has precisely one Hamilton circle.
We begin by stating the edge set of the subgraph $C$ defining the Hamilton circle $\overline{C}$ of $G$.
Let $E(C)$ consist of those edges of $E(G) \cap T^i_n$ for every $n$ and $i$ that correspond to the fat edges of $T$ in the most right picture of Figure~\ref{Tutte_HCs}.
Now consider any finite cut $D$ of~$G$.
The construction of $G$ yields that there exists an $N \in \mathbb{N}$ such that $D$ is already a cut of the graph obtained from $G_n$ by identifying the vertices corresponding to $u$, $\ell$ and $r$ of $T^1_0 \subseteq G_n$ for all~$n \geq N$.
Using this observation we can easily see that every vertex of $G$ has degree $2$ in $\overline{C}$.
We also obtain that every finite cut is met at least twice, but always in an even number of edges of $C$.
By Lemma~\ref{top_conn} we get that $\overline{C}$ is topologically and also arc-connected.
Therefore, every end of $G$ has edge-degree at least $1$ and at most~$3$ in~$\overline{C}$.
Together with Theorem~\ref{cycspace} this implies that every end of $G$ has edge-degree~$2$ in~$\overline{C}$.
Hence, Lemma~\ref{circ} tells us that $\overline{C}$ is a circle, which is Hamiltonian since it contains all vertices of $G$.

We finish the proof by showing that $\overline{C}$ is the unique Hamilton circle of $G$.
Since any Hamilton circle $\overline{H}$ of $G$ meets each cut $E(A^i_n, B^i_n)$ precisely twice, $\overline{H}$ induces a path through $T$ that contains all vertices of $T$ except one out of the set $\lbrace u, \ell, r \rbrace$.
By Lemma~\ref{HCs_in_T} we know that such paths must contain the edge adjacent to $u$.
Let us consider any $T^i_n$ in $G_n$.
Now let $T^j_{n+1}$ be the copy of $T$ whose vertices of degree~$1$ we have identified with the vertices corresponding to the neighbours of $c$ in $T^i_n$ during the construction of $G_{n+1}$.
The way we have identified the vertices implies that the path induced by $\overline{H}$ through $T^i_n$ must also use the edge adjacent to $\ell$ since the induced path in $T^j_{n+1}$ must use the edge adjacent to $u$.
With a similar argument we obtain that the induced path inside $T^i_n$ must use the edge corresponding to $vw$.
We know from Lemma~\ref{HCs_in_T} that there is a unique Hamilton path in $T-r$ that uses the edges~$\ell c$ and $vw$, namely the one corresponding to the fat edges in the most right picture of Figure~\ref{Tutte_HCs}.
So the edges which must be contained in every Hamilton circle are precisely those of $C$.
\end{proof}

\begin{remark}
\textnormal{After reading a preprint of this paper Max Pitz~\cite{pitz} carried further some ideas of this paper.
Also using the graph $T$, he recently constructed a two-ended cubic graph with a unique Hamilton circle where both ends have vertex- as well as edge-degree $3$.
He further proved that every one-ended Hamiltonian cubic graph whose end has edge-degree $3$ (or vertex-degree $3$) admits a second Hamilton circle.}
\end{remark}

\section*{Acknowledgement}

I would like to thank Tim R\"uhmann for reading an early draft of this paper and giving helpful comments.

\end{document}